\documentclass[12pt]{amsart}
\usepackage{color,psfrag,epsfig,url,hhline%,showkeys
}
\usepackage{tabu}
\numberwithin{figure}{section}
\numberwithin{table}{section}
\newtheorem{thmIntro}{Theorem}    

\newtheorem{theorem}{Theorem}[section]

\newtheorem{lemma}[theorem]{Lemma}
\newtheorem{prop}[theorem]{Proposition}
\newtheorem{conjecture}[theorem]{Conjecture}
\theoremstyle{definition}
\newtheorem{definition}[theorem]{Definition}

\newtheorem{notation}[theorem]{Notation}
\newtheorem{cor}[theorem]{Corollary}

\theoremstyle{remark}
\newtheorem{remark}[theorem]{Remark}

\numberwithin{equation}{section}
\newfont{\tap}{tap scaled 650}

\def \N{{\mathbb N}}

\def \Z{{\mathbb Z}}

\def \Q{{\mathbb Q}}

\def \[{[ }
\def \]{] }

\def\t{\widetilde}

\def \sgn{\mathrm{sgn }}
\def \gcd{\mathrm{gcd}}
\def \op{^\mathrm{op}}

\definecolor{dgreen}{rgb}{0,0.5,0}        
\definecolor{dred}{rgb}{0.5,0,0}

\begin{document}

\title{Mutation-finite quivers with real weights}

\author{Anna Felikson and Pavel Tumarkin}
\address{Department of Mathematical Sciences, Durham University, Mathematical Sciences \& Computer Science Building, Upper Mountjoy Campus, Stockton Road, Durham, DH1 3LE, UK}
\email{anna.felikson@durham.ac.uk, pavel.tumarkin@durham.ac.uk}
\thanks{Research of P.T. was supported in part by the Leverhulme Trust research grant RPG-2019-153.}

%\author{Michael Shapiro}
%\address{Department of Mathematics, Michigan State University, East Lansing, MI 48824, USA}
%\email{mshapiro@math.msu.edu}

%\author{Pavel Tumarkin} 
%\address{Department of Mathematical Sciences, Durham University, Science Laboratories, South Road, Durham, DH1 3LE, UK}
%\email{pavel.tumarkin@durham.ac.uk}

\begin{abstract}
We classify all mutation-finite quivers with real weights. We show that every finite mutation class not originating from an integer skew-symmetrizable matrix has a geometric realization by reflections. We also explore the structure of acyclic representatives in finite mutation classes and their relations to acute-angled simplicial domains in the corresponding reflection groups.   
\end{abstract}

\maketitle
\setcounter{tocdepth}{1}
\tableofcontents

\section{Introduction and main results}
Mutations of quivers were introduced by Fomin and Zelevinsky in~\cite{FZ1} in relation to cluster algebras. Mutations are involutive transformations decomposing the set of quivers into equivalence classes called {\em mutation classes}. Of special interest are quivers of {\em finite mutation type}, i.e. those whose mutation classes are finite, these quivers have shown up recently in various contexts.
Most of such quivers are adjacency quivers of triangulations of marked bordered surfaces~\cite{FG,GSV,FST,FT}, the complete classification of mutation-finite quivers was obtained in~\cite{FeSTu1}.

In this paper, we consider a more general notion of a quiver -- we allow arrows of quivers to have {\em real weights} (and we refer to ``usual'' quivers as to {\em integer quivers}). Quivers with real weights of arrows have been studied in~\cite{BBH}, where the Markov constant was used to explore the mutation classes of quivers of rank $3$. Quivers originating from non-crystallographic finite root systems were also considered in~\cite{L}. A categorification of mutations of quivers of non-crystallographic types $I_2(2n+1), H_3,H_4$ was constructed in~\cite{DT}. In~\cite{FT-rk3} we constructed a geometric model for mutations of rank $3$ quivers with real weights and classified all finite mutation classes. The main result of this paper is a classification of all finite-mutation quivers. More precisely, we prove the following theorem.

\begin{thmIntro}%[Theorem~\ref{}]
\label{intro1}
For every  mutation-finite  non-integer quiver $Q$ of rank $r\ge 3$  
\begin{itemize}
\item[-]  either $Q$ arises from a triangulated orbifold; 
\item[-]  or $Q$ is mutation-equivalent to one of the $F$-type quivers shown in Fig.~\ref{4};
\item[-]  or $Q$ is mutation-equivalent to one of the $H$-type quivers  shown in Fig.~\ref{5};
\item[-]  or $Q$ is mutation-equivalent to a representative of one of the three series of quivers shown in Fig.~\ref{ser}.
\end{itemize}
\end{thmIntro}

We list all non-orbifold mutation-finite classes and their sizes in Tables~\ref{table mut-fin} and~\ref{table sizes mut-fin} respectively. Notice that the list above includes the appropriate {\em rescalings} (see~\cite[Section~7]{R})  of all mutation-finite diagrams from~\cite{FeSTu2}:  quivers arising from triangulated orbifolds~\cite{FeSTu3} and $F$-type quivers are explicitly mentioned in Theorem~\ref{intro1}, and $G_2$-type quivers obtained from the diagrams $G_2^{(*,+)}$ and   $G_2^{(*,*)}$ belong to the series mentioned in the last line of Theorem~\ref{intro1}.

%\end{remark}

Our proof of Theorem~\ref{intro1} is based on the classification of mutation-finite rank $3$ quivers~\cite{FT-rk3} and the related geometry. In particular, all the weights of arrows of mutation-finite quivers should be of the form $2\cos({q\pi/d})$ for some integer $q$ and $d$, and every rank $3$ subquiver has to correspond to some spherical or Euclidean triangle (we recall the details in Section~\ref{3}). We first show that all mutation-finite quivers of sufficiently high rank originate from orbifolds, and then treat quivers in low ranks where we obtain three exceptional infinite families.

Next, we construct a geometric realization for every finite mutation class of quivers except for mutation-cyclic ones originating from orbifolds (we conjecture though that mutation classes originating from unpunctured orbifolds also have geometric realizations by reflections). The realization by reflections provides a quadratic form associated to a mutation class, and thus we can define quivers of {\em finite type} when the corresponding quadratic form is positive definite. As in the integer case (see~\cite{FZ2}), these correspond precisely to finite reflection groups. We say that a quiver has {\em affine} or {\em extended affine} type if its mutation class is realized in a semi-positive quadratic space of corank $1$ (at least $2$, respectively). The result can be then formulated in the following theorem.

\begin{thmIntro}%[Theorem~\ref{realisation}]
\label{intro2}
Every non-integer finite mutation class (except for the quivers originating from orbifolds) has a geometric realization by reflections. In particular,  
\begin{itemize}
\item[-] quivers in the top row of Table~\ref{table mut-fin} are of finite type;
\item[-] quivers in the middle row of  Table~\ref{table mut-fin} are of affine type; 
\item[-] quivers in the bottom row of  Table~\ref{table mut-fin} are of extended affine type.
\end{itemize}
\end{thmIntro}

\begin{center}
\begin{table}
\caption{Mutation-finite non-integer non-orbifold type quivers}
\label{table mut-fin}
\epsfig{file=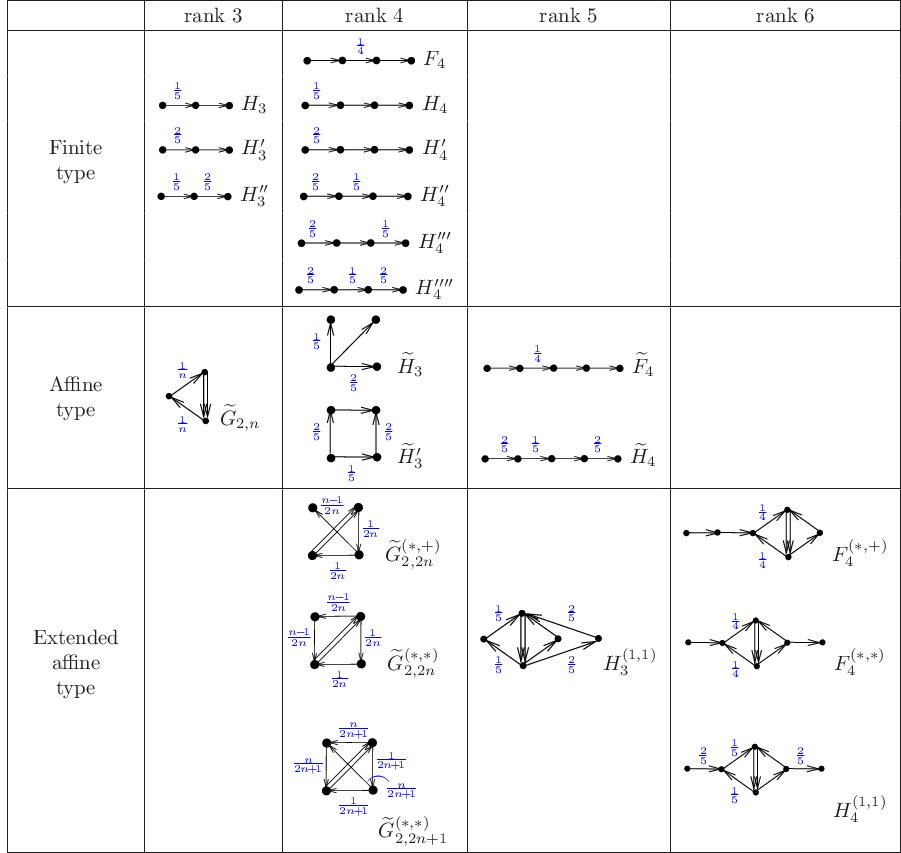,width=0.99\linewidth}
\end{table}  
\end{center}

\bigskip

\begin{center}
\begin{table}
  \caption{Sizes of mutation classes of type $H$ and type $F$ quivers}
\label{table sizes mut-fin}
\epsfig{file=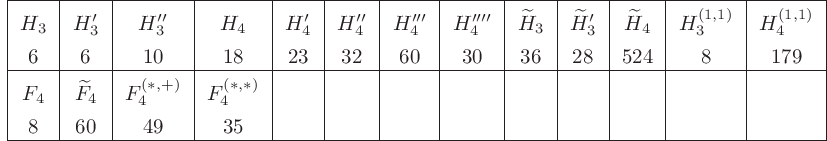,width=0.86\linewidth}
\end{table}  
\end{center}

The paper is organized as follows. In Section~\ref{3} we recall some details from~\cite{FT-rk3} on the classification of mutation-finite quivers in rank $3$ which will be our main tool. In Section~\ref{sec-high} we show that in rank greater than $4$  the denominator $d$ in the weight $2\cos({q\pi/d})$ of an arrow of a mutation-finite quiver is bounded by $5$. Thus, we restrict our considerations to quivers with weights of arrows belonging to $\Z[\sqrt{2}]$ and  $\Z[(1+\sqrt{5})/2]$ which we consider in Section~\ref{low denominator}, and to quivers of rank $4$ considered in Section~\ref{4high}. In Section~\ref{sec real} we show that mutations in finite mutation classes can be modelled by partial reflections in positive semi-definite quadratic spaces. Finally, in Section~\ref{sec acyclic} we explore the relations between acyclic representatives in finite mutation classes and acute-angled simplices bounded by mirrors of reflections.

%\begin{center}
%\begin{table}
%\caption{Sizes of mutation classes of type $F$   quivers}
%\label{table sizes mut-fin}
%\begin{tabular}{|c|c|c|c|}
%\hline
%   $F_4$ &  $\widetilde F_{4}$ & $F_4^{(*,+)}$ & $F_4^{(*,*)}$\\
%    8    &     60             &      49         &   35          \\      
%\hline
%\end{tabular}  
%\end{table}  
%\end{center}

\section{Classification in rank $3$}
\label{3}
In this section we recall the results of~\cite{FT-rk3} this paper is based on, and deduce some straightforward corollaries we will use throughout the text. We start with reminding the notation we used in~\cite{FT-rk3} and introducing some new one.

\begin{notation}
\begin{itemize} 
\item[]
\item Given a quiver $Q$ with vertices $1,\dots,n$, and a subset $I\subset \{1,\dots,n\}$, denote by $Q_I$ the subquiver of $Q$ spanned by vertices $\{ i\in I\}$. In particular, the vertex labeled $i$ will be denoted $Q_i$. 
  For example, $Q_{124}$ will denote a subquiver spanned by vertices $Q_1,Q_2,Q_4$.
\item While drawing quivers we will use the following conventions:
 \begin{itemize}
  \item given an arrow $Q_{ij}$ of weight $2\cos \frac{\pi m}{d}$, we will label this arrow  by $\frac{m}{d}$; we will also say that $Q_{ij}$ is an arrow {\it marked  $\frac{m}{d}$}; %Writing this we will assume that $(m,k)=1)$. 
  \item arrows of weight $1$ will be left unlabeled;
  \item we draw double arrows instead of arrows of weight $2$.      
 \end{itemize}  
\item We say that an arrow $Q_{ij}$ {\em vanishes} if $Q_i$ and $Q_j$ are not joined in $Q$.
\item By $(a,b,c)$, $a,b,c>0$ we denote a rank $3$ cyclic quiver with arrows of weights $a,b,c$.
\item A rank $3$ acyclic quiver with arrows of weights $a,b$ looking in one direction and an arrow of weight $c$ in the opposite direction will be denoted by $(a,b,-c)$, where some weights may equal $0$.  
\item Given a quiver $Q$, we will denote by $Q\op$ the quiver obtained from $Q$ by reversing all arrows. $Q\op$ is also called a quiver {\em opposite} to $Q$. 
\end{itemize}
\end{notation}

\begin{theorem}[\cite{FT-rk3}, Theorem 6.11]
\label{finite-thm}
Let $Q$ be a connected rank $3$ quiver with real weights. Then $Q$ is of finite mutation type if and only if it is mutation-equivalent
to one of the following quivers:
\begin{itemize}
\item[(1)] $(2,2,2)$;
\item[(2)]  $(2,2\cos\frac{\pi}{d},2\cos\frac{\pi}{d})$,  $d\in \Z_+$;
\item[(3)] $(1,1,0)$, $(1,\sqrt 2,0)$, $(1,2\cos\frac{\pi}{5},0)$, $(2\cos\frac{\pi}{5},2\cos \frac{2\pi}{5},0)$, $(1,2\cos \frac{2\pi}{5},0)$.
\end{itemize}

\end{theorem}

%\begin{remark}
%Notice that $1=2\cos\frac{\pi}{3}, \sqrt 2=2\cos\frac{\pi}{4}$ and $2=2\cos 0$.
%\end{remark}  

Below we list some corollaries of Thm.~\ref{finite-thm} and related geometric constructions proved in~\cite{FT-rk3}. 

\begin{cor}
\label{cor3}
Let $Q$ be a connected mutation-finite rank $3$ quiver with real weights. Then
\begin{itemize}
\item[(1)] All weights of $Q$ are  of the form $2\cos\frac{\pi m}{d}$ with $m,d\in \Z$, $m\le d/2$.
\item[(2)] If $Q$ contains an arrow marked $\frac{m}{d}$ with $d>5$, $\gcd(m,d)=1$, then $Q$ is mutation-equivalent to  
 $(2,2\cos\frac{\pi}{d},2\cos\frac{\pi}{d})$.
\item[(3)]  If $Q$ contains an edge of weight $2$ then $Q$ is a cyclic quiver which coincides with either $(2,2,2)$ or  
  $(2,2\cos\frac{\pi m}{d},2\cos\frac{\pi m}{d})$, $d,m,\in \N$, $0<m\le d/2$.
\item[(4)] If $Q=(a,b,-c)$, $a,b,c>0$ is an acyclic quiver,  then 
\begin{itemize}
\item[-] 
$Q=(2\cos\frac{\pi m}{d},2\cos\frac{\pi s}{d}, -2\cos\frac{\pi t}{d})$
 for some  $d,m,s,t\in \N$, $m,s,t\le d/2$ such that $\frac{ m}{d}+\frac{ s}{d}+\frac{ t}{d}\ge 1$; 
\item[-] 
if in addition at least one of $a,b,c$ equals $2\cos\frac{\pi m}{d}$ with $d>5$, $\gcd(m,d)=1$,  \\ then 
 $\frac{ m}{d}+\frac{ s}{d}+\frac{ t}{d}=1$.
\end{itemize}

\item[(5)]  If $Q=(a,b,c)$, $a,b,c>0$ is a cyclic quiver then  
\begin{itemize}
\item[-]
$Q=(2\cos\frac{\pi m}{d},2\cos\frac{\pi s}{d}, 2\cos\frac{\pi t}{d})$
 for some  $d,m,s,t\in \N$, $m,s,t\le d/2$
such that $\frac{ m}{d}+\frac{ s}{d}+(1-\frac{ t}{d})\ge 1$ (up to permutation of $m,s,t$); 
 \item[-] 
if in addition at least one of $a,b,c$ equals to  $2\cos\frac{\pi m}{d}$ with $d>5$, $\gcd(m,d)=1$, \\ then 
 $\frac{ m}{d}+\frac{ s}{d}+(1-\frac{ t}{d})=1$.
\end{itemize}

%\item[(6)] if $Q$ contains an arrow of weight $2$ then $Q$ is cyclic.

\end{itemize}

\end{cor}

The equalities $(4)$--$(5)$ in Corollary~\ref{cor3} have a geometric interpretation: for every mutation-finite quiver there is a spherical or Euclidean triangle with the corresponding angles (the triangle is acute-angled if $Q$ is acyclic, and has an obtuse angle otherwise). We also remind that the mutations can be modeled by partial reflections, see~\cite{FT-rk3}.

\section{High denominators in ranks $5$ and higher}
\label{sec-high}
In this section we show that there are no high denominator quivers of rank higher than $4$ (Theorem~\ref{no big}).
To prove the theorem, we start with several technical lemmas (Lemma~\ref{no markov}-\ref{no-ac-high}) about rank $3$ and $4$ quivers. 

\begin{definition}
Given a quiver $Q$, we say that $d\in\N$ is the {\em highest denominator} in the mutation class of $Q$ if all weights of quivers in the mutation class of $Q$ are either $2$ or of the form $2\cos \frac{p'}{d'}$ with $d'\le d$, and there exists a quiver $Q'$ in the mutation class of $Q$ with an arrow of weight $2\cos \frac{p}{d}$, $\gcd(p,d)=1$. Abusing notation, we will say that $Q$ is a {\em denominator $d$ quiver}.
\end{definition}

\begin{lemma}
\label{no markov}
No connected mutation-finite quiver contains the Markov quiver $(2,2,2)$ as a proper subquiver.   

\end{lemma}

\begin{proof}
Suppose the contrary,  i.e. $Q=Q_{1234}$ with  $Q_{123}=(2,2,2)$ is a mutation-finite quiver.
By Corollary~\ref{cor3}(3) all rank 3 subquivers of $Q$ should be cyclic, which is clearly impossible.  
\end{proof}

%\begin{figure}[!h]
%\psfrag{2_}{\scriptsize  $2$}
%\psfrag{1}{\scriptsize  $1$}
%\psfrag{2}{\scriptsize  $2$}
%\psfrag{3}{\scriptsize  $3$}
%\psfrag{4}{\scriptsize  $4$}
%\psfrag{1k}{\scriptsize  $\theta_{d}$}
%\epsfig{file=./pic/markov_.eps,width=0.3\linewidth}
%\caption{To the proof of Lemma~\ref{n and markov}}
%\label{and Markov}
%\end{figure}\texttt{}

%\begin{lemma}
%Let $Q=(\theta_{1,k},\theta_{l,k},-\theta_{l,k})$ be an acyclic quiver. 
%Then $Q$ is not mutation finite. 
%\end{lemma}
%
%\begin{proof}
%Let $\mu$ be a non-sink/source mutation of $Q$. 
%Then $\mu(Q)$ is a cyclic quiver with weights $\theta_{1,k}$, $\theta_{l,k}$ and $\theta_{m,k}$ such that    $=(\theta_{1,k},\theta_{l,k},2)$ 
%
%
%\end{proof}

\begin{figure}[!h]
\psfrag{II}{\small  $I\!I$}
\psfrag{1}{\small $\frac{1}{2n}$}
\psfrag{n-1}{\small $\frac{n-1}{2n}$}
\psfrag{1_}{\small $\frac{1}{2n+1}$}
\psfrag{n}{\small $\frac{n}{2n+1}$}
\epsfig{file=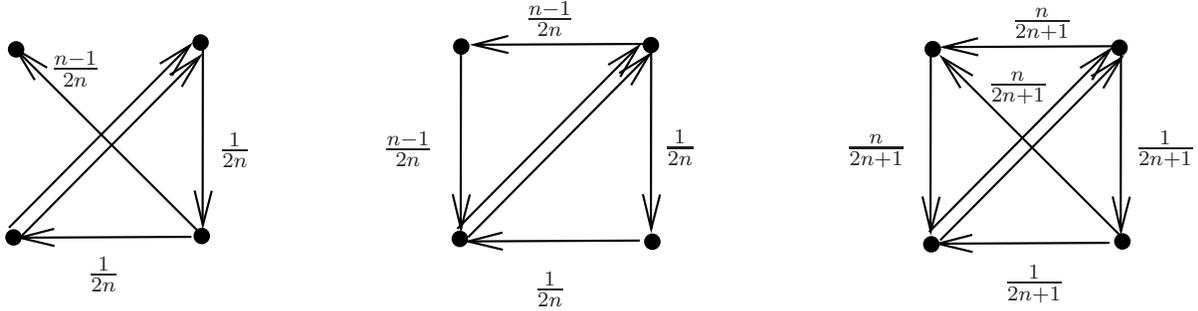,width=0.99\linewidth}
\caption{Three infinite series of  quivers. Following notation introduced in Section~\ref{3}, the arrows of weight $2\cos \frac{\pi p}{q}$ are labeled by $\frac{p}{q}$, and the arrows of weight $2$ are shown by double arrows. }
\label{ser}
\end{figure}

\begin{lemma}
\label{thm: ser}
Let $Q$ be a connected mutation-finite quiver of rank $4$. Suppose that $d>5$ is the highest denominator of weights of arrows in the mutation class of $Q$. 
Then either $Q$ or $Q\op$ is mutation-equivalent to  one of the quivers listed in Fig.~\ref{ser}.
%In particular, every arrow in $Q$ are marked by $\frac{p}{d}$ for some $p\le \frac{d}{2}$.
\end{lemma}

\begin{proof}
Consider such a quiver $Q$. We can assume that $Q$ has an arrow of weight  $2\cos\frac{\pi m}{d}$, $d>5$, and $m$ and $d$ are coprime.
Let $Q_{123}$ be a rank $3$ connected subquiver of $Q$ containing an arrow of weight  $2\cos\frac{\pi m}{d}$. Then $Q_{123}$ corresponds to a Euclidean triangle, and hence, the mutation class of $Q_{123}$
contains an oriented subquiver with weights $(2,2\cos\frac{\pi}{d},2\cos\frac{\pi}{d})$ (without loss of generality we can assume that this is the subquiver $Q_{123}$ itself, and $Q_{13}$ is the double arrow).
Moreover, as a mutation-finite rank $3$ subquiver $Q_{134}$ with a double arrow cannot be a Markov quiver (see Lemma~\ref{no markov}), $Q_{134}$  should be a cyclic quiver with the weights
$(2,2\cos\frac{\pi p}{d'},2\cos\frac{\pi p}{d'})$) for some $p\le d'/2$, $d'\le d$ (see Corollary~\ref{cor3}(4)).  We conclude that $Q$ is
%mutation-equivalent to
the quiver 
shown in Fig.~\ref{pf1}, where the weight of $Q_{24}$ is $2\cos \frac{\pi q}{d''}$ with  $q\le d''/2$, $d''\le d$ 
(note that the arrow $Q_{24}$ may be oriented in the opposite way -- this would mean the quiver in Fig.~\ref{pf1} is the opposite quiver $Q\op$).

\begin{figure}[!h]
\psfrag{II}{\small  $I\!I$}
\psfrag{q}{\small $\frac{q}{d''}$}
\psfrag{1_}{\small $\frac{1}{d}$}
\psfrag{p}{\small $\frac{p}{d'}$}
\psfrag{1}{\small $1$}
\psfrag{2}{\small $2$}
\psfrag{3}{\small $3$}
\psfrag{4}{\small $4$}
\epsfig{file=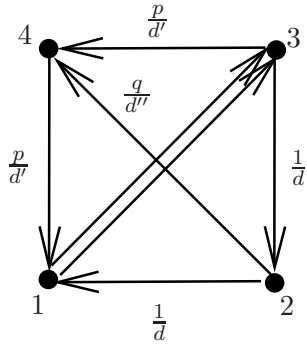,width=0.3\linewidth}
\caption{To the proof of Lemma~\ref{thm: ser}.}
\label{pf1}
\end{figure}

Consider the acyclic subquiver $Q_{234}$: as it contains an arrow $Q_{23}$ of weight $2\cos\frac{\pi}{d}$ with $d>5$, Corollary~\ref{cor3}(4) implies
\begin{equation}
\label{formula}  
\frac{1}{d}+\frac{p}{d'}+\frac{q}{d''}=1.
\end{equation}

We will consider three cases: either $p/d'=1/2$ (i.e. the arrows $Q_{34}$ and $Q_{14}$ vanish), or $q/d''=1/2$ (the arrow $Q_{24}$ vanishes), or otherwise, all six arrows are present in $Q$.

\noindent
{\bf Case 1. } If $\frac{p}{d'}=\frac{1}{2}$, then  $\frac{q}{d''}=\frac{1}{2}-\frac{1}{d}$. Since $d''\le d$, we conclude that $d=2n$ for some $n\in \N$, which implies
$d''=2n$, $q=n-1$, and  $Q$ is the quiver shown in the middle of  Fig.~\ref{ser}.

\noindent
{\bf Case 2.} If  $\frac{q}{d''}=\frac{1}{2}$, then  $\frac{p}{d'}=\frac{1}{2}-\frac{1}{d}$, so that  $d=d'=2n$, and $p=n-1$ (where $n\in \N$), which 
produces the quiver on the left of  Fig.~\ref{ser}.

\noindent
{\bf Case 3.} Suppose that $\frac{p}{d'}\ne \frac{1}{2} \ne  \frac{q}{d''}$. This implies that $\frac{p}{d'}< \frac{1}{2}$ and  $\frac{q}{d''}< \frac{1}{2}$. Moreover, as $d'\le d$ and $d''\le d$, we see that $\frac{1}{2}-\frac{p}{d'}\ge \frac{1}{2d}$ and  $\frac{1}{2}-\frac{q}{d''}\ge \frac{1}{2d}$. In view of~(\ref{formula}) this implies  $\frac{1}{2}-\frac{p}{d'}= \frac{1}{2d}$ and  $\frac{1}{2}-\frac{q}{d''}= \frac{1}{2d}$.
Hence, $d=d'=d''=2n+1$ for some $n\in\N$, $p=q=n$,
and $Q$ is the quiver shown on the right of  Fig.~\ref{ser}.
\end{proof}

\begin{remark}
\label{necessary}
Notice that Lemma~\ref{thm: ser} gives a {\em necessary} condition for a quiver of rank $4$ with large denominator to be mutation-finite. We will show that this condition is also sufficient (i.e., all quivers in Fig.~\ref{ser} are indeed mutation-finite) in Section~\ref{4high}.  
\end{remark}

The following lemma can be verified by a straightforward computation.
\begin{lemma}
\label{l-1,k}
Let  $d=2n+1$ where $n\ge 2$, $n\in \N$. Let $Q$ be an acyclic quiver of rank $3$ and $\mu$ be a non-sink/source mutation of $Q$.
Then
\begin{itemize}
\item[(a)]
  if $Q=(2\cos\frac{\pi}{d},2\cos\frac{\pi n}{d},-2\cos\frac{\pi n}{d})$, 
  then $\mu(Q)=(2\cos\frac{\pi n }{d},2\cos\frac{\pi n}{d},2\cos\frac{\pi (n-1)}{d})$;
\item[(b)]
  if $Q=(2\cos\frac{\pi n}{d},2\cos\frac{\pi n}{d},-2\cos\frac{\pi }{d})$, 
  then $\mu(Q)=(2\cos\frac{\pi n }{d},2\cos\frac{\pi n}{d},2)$.
\end{itemize}
\end{lemma}

\begin{remark}
The quiver $Q$ in Lemma~\ref{l-1,k} corresponds to an acute-angled  Euclidean triangle $T$ with angles $(\frac{\pi}{d},\frac{\pi n}{d},\frac{\pi n}{d})$, so the statement can also be easily checked by applying partial reflections.
\end{remark}

\begin{lemma}
\label{no-ac-high}  
Let $Q=Q_{1234}$ be an acyclic connected rank $4$ quiver. Suppose that the vertex $Q_3$ is not joined with $Q_1$ in $Q$.
If the weight of $Q_{12}$ is $2\cos \frac{\pi m}{d}$ with $d>5$, $\gcd(m,d)=1$,  then $Q$ is mutation-infinite.

%Then
%\begin{itemize}
%\item[(1)] if the weight of $Q_{12}$ is $2\cos \frac{\pi m}{d}$ with $k>5$, $(m,k)=1$,  then $Q$ is mutation infinite;
%\item[(2)] if weights of all non-vanishing arrows of $Q$ are  $\sqrt 2$ then
%  $Q$ is mutation infinite.
%
%\end{itemize}  

\end{lemma}  

\begin{proof}
Suppose that $Q$ is mutation-finite, and assume first that $Q_3$ is neither a sink nor a source. In particular, it is connected to both $Q_2$ and $Q_4$. Then, as $Q$ is  acyclic, the mutation $\mu_3$ at vertex $Q_3$ changes the weight of the arrow $Q_{24}$ but does not change its direction.

Now, consider the subquiver $Q_{124}$. Since  the arrow incident to $Q_{12}$ has the weight $2\cos \frac{\pi m}{d}$  with $d>5$, % (or all non-vanishing arrows are of the weight $2\cos \frac{\pi }{4}=\sqrt 2$),
we conclude that the acyclic subquiver $Q_{124}$ can be modeled by an acute-angled Euclidean triangle, and moreover, the weight of the arrow $Q_{24}$ is uniquely determined by the weights of $Q_{12}$ and $Q_{14}$.
Since $Q_3$ is not joined with $Q_1$, the mutation $\mu_3$ preserves the weights and directions of arrows $Q_{12}$ and $Q_{14}$. Since $\mu_3$ also preserves the direction of $Q_{24}$, this implies that the subquiver $Q'_{124}$ of the mutated quiver $Q'=\mu_3(Q)$ is still acyclic and satisfies the same properties as $Q_{124}$: it is modeled by an acute-angled Euclidean triangle. Hence, the weight of the new arrow $Q_{24}'$ should coincide with the weight of the old arrow $Q_{24}$.
This contradicts the result of the paragraph above. The contradiction shows that $Q$ is mutation-infinite.

Assume now that $Q_3$ is either a sink or a source. We will now show that applying sink/source mutations only we can make $Q_3$ neither a sink nor a source, and thus reduce the case to the one already being considered.

Indeed, without loss of generality we can assume $Q_3$ is a sink. Applying, if necessary, a source mutation in $Q_1$, we can assume that $Q_1$ is not a source. Since $Q$ is acyclic, it contains a source, and thus either $Q_2$ or $Q_4$ is a source. After mutating at a source, the vertex  $Q_3$ is neither a sink nor a source anymore, so we are in the assumptions of the first case.    

\end{proof}

\begin{theorem}
\label{no big}
There is no rank $5$ connected mutation-finite quiver with an arrow of weight $2\cos \frac{\pi m}{d}$ 
with $d>5$, $m\le d/2$, $\gcd(m,d)=1$.

\end{theorem}

\begin{proof}
Suppose that $Q$ is a mutation-finite connected  rank $5$ quiver, and assume that $Q_{1234}$  is a connected subquiver containing an arrow  of weight   $2\cos \frac{\pi m}{d}$, $d>5$, $\gcd(m,d)=1$. We can assume that $d$ is the highest denominator in the mutation class of $Q$.
Then by Lemma~\ref{thm: ser}, $Q_{1234}$ is mutation-equivalent to one of the quivers in Fig.~\ref{ser} or its opposite. Without loss of generality we may assume that the subquiver $Q_{1234}$ of $Q$ itself is  one of the quivers in   Fig.~\ref{ser}. We consider these three series of quivers separately.

\medskip
\noindent
{\bf Case~1: Odd denominator $d=2n+1$.} Then $Q_{1234}$ is the subquiver shown in Fig.~\ref{ser} on the right (we can assume $Q_{13}$ is the double arrow and $Q_2$ is the vertex incident to two arrows marked $\frac{1}{d}$).
By reasoning as in Case~3 of the proof of  Lemma~\ref{thm: ser}, we see that the subquiver $Q_{1235}$ looks identical to $Q_{1234}$ modulo the direction of the arrow $Q_{25}$ which can point either way (see Fig.~\ref{5-odd}(a)~and~(b)). Applying  Corollary~\ref{cor3}(4) to acyclic subquiver $Q_{145}$ we see that the arrow $Q_{45}$ should be marked $\frac{1}{d}$. In the case shown in Fig.~\ref{5-odd}(b) we also see that the arrow $Q_{45}$ is directed from $Q_4$ to $Q_5$  (as the weights of arrows in the subquiver $Q_{245}$ require this subquiver to be acyclic); in the case shown in Fig.~\ref{5-odd}(a) the vertices $Q_4$ and $Q_5$ are completely symmetric, so we can also assume $Q_{45}$ is directed from $Q_4$ to $Q_5$. Therefore, the quiver $Q$ is one of the two quivers shown on the left of Fig.~\ref{5-odd}.
Applying mutation $\mu_5$ in vertex $Q_5$, we obtain the quiver $Q'=\mu_5(Q)$ shown on the right of Fig.~\ref{5-odd} (we use Lemma~\ref{l-1,k} to compute the new weights of arrows). However, the subquiver $Q_{234}'$ of $Q'$ is an acyclic subquiver
 with  arrow of weight $2$, which is impossible by   Corollary~\ref{cor3}(3).

\begin{figure}[!h]
\psfrag{a}{\small (a)}
\psfrag{b}{\small (b)}
\psfrag{1}{\scriptsize  $1$}
\psfrag{2}{\scriptsize  $2$}
\psfrag{3}{\scriptsize  $3$}
\psfrag{4}{\scriptsize  $4$}
\psfrag{5}{\scriptsize  $5$}
\psfrag{m5}{\scriptsize  $\mu_5$}
\psfrag{1k}{\scriptsize  $\frac{1}{d}$}
\psfrag{lk}{\scriptsize  $\frac{n}{d}$}
\psfrag{l-1}{\scriptsize  $\frac{n\!-\!1}{d}$}
\epsfig{file=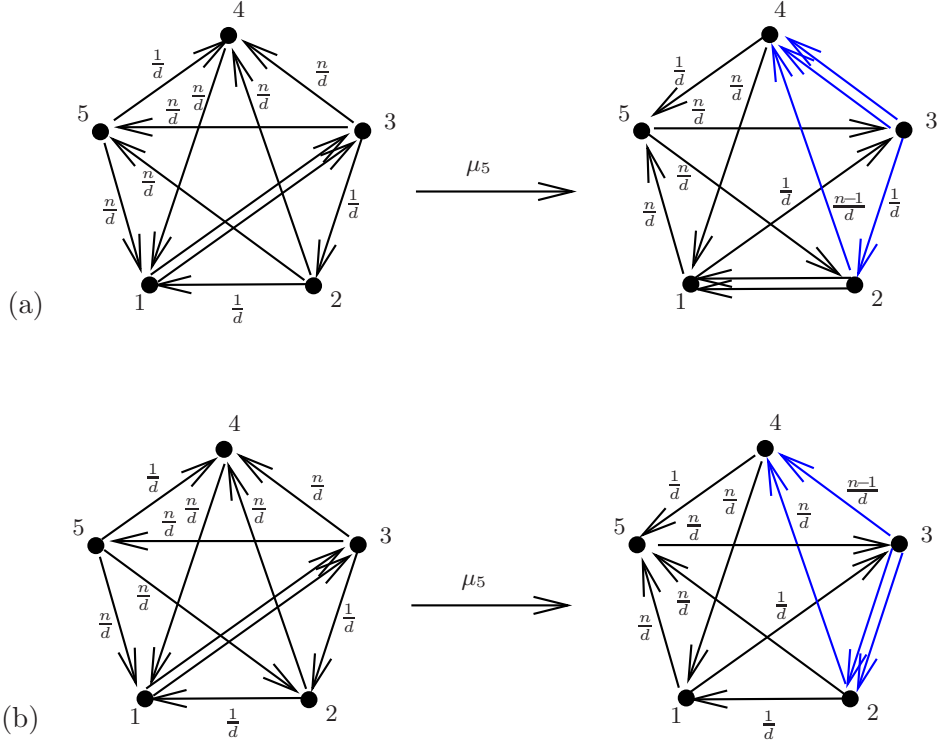,width=0.8\linewidth}
\caption{Quiver of rank $5$ with an arrow of weight $2\cos \frac{\pi m}{d}$ for odd $d=2n+1$.}
\label{5-odd}
\end{figure}

\medskip
\noindent
{\bf Case~2: Even denominator $d=2n$.}
In this case there are two possibilities for each of the subquivers $Q_{1234}$ and $Q_{1235}$ (see Fig.~\ref{ser}), which, up to symmetry and taking $Q\op$ (and sink/source mutations), gives rise to four forms of the quiver $Q$ shown in Fig.~\ref{5-even}. In each of the four possibilities the weight of the arrow $Q_{45}$ is determined uniquely from subquivers $Q_{145}$ or $Q_{245}$.

\begin{figure}[!h]
\psfrag{a}{\small  (a)}
\psfrag{b}{\small  (b)}
\psfrag{c}{\small  (c)}
\psfrag{d}{\small  (d)}
\psfrag{1}{\scriptsize  $1$}
\psfrag{2}{\scriptsize  $2$}
\psfrag{3}{\scriptsize  $3$}
\psfrag{4}{\scriptsize  $4$}
\psfrag{5}{\scriptsize  $5$}
\psfrag{1k}{\scriptsize  $\frac{1}{d}$}
\psfrag{2k}{\scriptsize  $\frac{2}{d}$}
\psfrag{lk}{\scriptsize  $\frac{n}{d}$}
\psfrag{l-1}{\scriptsize  $\frac{n\!-\!1}{d}$}
\epsfig{file=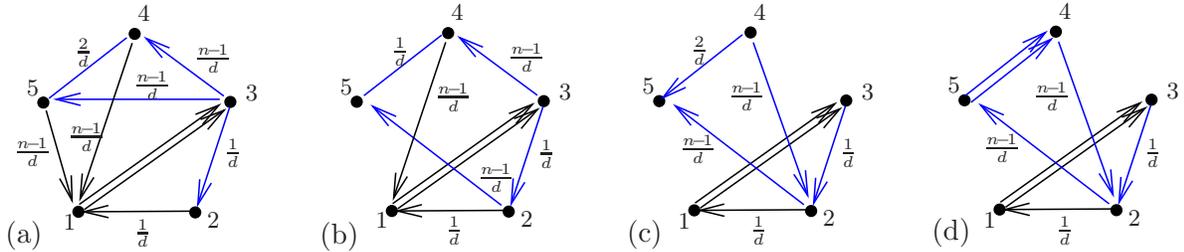,width=0.99\linewidth}
\caption{Quiver of rank $5$ with an arrow of weight $2\cos \frac{\pi m}{d}$ for even $d=2n$.}
\label{5-even}
\end{figure}

 Notice that in cases (a), (b) and (c) the subquiver $Q_{2345}$ is acyclic, having a vertex ($Q_2$, $Q_3$ and $Q_3$  in the three cases respectively) which is not joined with $Q_4$ and incident to the arrow $Q_{23}$ of weight $2\cos \frac{\pi}{d}$, $d>5$. So, by Lemma~\ref{no-ac-high} $Q_{2345}$ (and, hence, $Q$) is mutation-infinite.

We are left to consider the case (d). Applying mutations in vertices $Q_2$ and the $Q_1$, we obtain the quiver $Q'$ shown in Fig.~\ref{d}. Its subquiver $Q'_{245}$ is acyclic, has a denominator $d>5$ arrow $Q'_{25}$, but does not correspond to a Euclidean triangle, so it is mutation-infinite.

\begin{figure}[!h]
\psfrag{1}{\scriptsize  $1$}
\psfrag{2}{\scriptsize  $2$}
\psfrag{3}{\scriptsize  $3$}
\psfrag{4}{\scriptsize  $4$}
\psfrag{5}{\scriptsize  $5$}
\psfrag{1k}{\scriptsize  $\frac{1}{d}$}
\psfrag{2k}{\scriptsize  $\frac{2}{d}$}
\psfrag{lk}{\scriptsize  $\frac{n}{d}$}
\psfrag{l-1}{\scriptsize  $\frac{n\!-\!1}{d}$}
\psfrag{l-2}{\scriptsize  $\frac{n\!-\!2}{d}$}
\epsfig{file=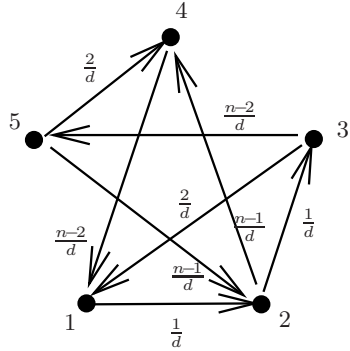,width=0.31\linewidth}
\caption{A mutation of the quiver shown in Fig.~\ref{5-even}(d).}
\label{d}
\end{figure}

\end{proof}

\begin{remark}
In view of  Theorem~\ref{no big}, in rank $5$ and higher we  only need to consider the quivers with arrows marked $\frac{p}{d}$, $p<d\le 5$. This will be done in Section~\ref{low denominator}.

\end{remark}

\section{Low denominator quivers}
\label{low denominator}

By {\it low denominator} quivers we mean quivers without arrows marked $\frac{m}{d}$, where $m\le d/2$, $\gcd(m,d)=1$  and $d> 5$. 
There are finitely many of low denominator quivers in each rank, so one can classify mutation-finite  low denominator quivers of small ranks checking them case by case. 

%We will say that $Q$ is a {\it denominator $d$} quiver if $d$ is the maximal denominator of weights in the mutation class of $Q$.
%all weights  $\theta_{m',k'}$ of arrows in $Q$ satisfy $k'\le k$. 

%By a {\it common denominator of $Q$} we will call a number $D$ such that $d'|D$ for every $d'$ such that  $\frac{m'}{d'}$ is a mark for some arrow in $Q$.

%\begin{prop} Let $Q$ is a low denominator quiver.
%\begin{itemize}
%\item[(a)]
%If $rk(Q)\ge ??? $ then the denominator $k$ for any weight $\theta_{p,k}$ in $Q$ satisfies 
%$k\le 4$. 
%\item[(b)] k=5 mutation classes are listed in Table
%\item[(c)]  If $k\le 4$  for every weight $\theta_{p,k}$ in $Q$, then  
% $Q$ is a skew-symmetrization of an integer mutation finite matrix (with non-integer skew-symmetriser).  
%
%\end{itemize}
%\end{prop}

\subsection{Denominator $4$ mutation classes}

For every skew-symmetrizable integer matrix $B$ one can construct a skew-symmetrization $B'$ of it by putting $b'_{ij}=\sgn\,{b_{ij}}\sqrt{-b_{ij}b_{ji}}$. Matrix $B'$ gives rise to a (possibly non-integer) quiver $Q'$ whose mutations agree with mutations of the diagram of $B$ (see~\cite{FZ2}). Notice that not every non-integer denominator $4$ quiver corresponds to a diagram of an integer skew-symmetrizable matrix: to have a corresponding skew-symmetrizable matrix the number of arrows of weight $\sqrt{2}$ in every (not obligatory oriented) cycle must be even (cf.~\cite[Exercise 2.1]{K}). However, it is easy to check that any chordless cycle with odd number of arrows of weight $\sqrt{2}$ is mutation-infinite, and thus we can conclude that  
%Mutation rule for non-integer quivers coincides with the mutation rule for diagrams (constructed from skew-symmetrizable matrices as in~\cite{FST}), this is easy to check directly in rank $3$.
the finite mutation classes of denominator $4$ quivers are the same as the ones described in~\cite{FeSTu2}.

\begin{remark}
Denominator $3$ and $2$ quivers are actually integer, so we do not need to consider them.

\end{remark}

\begin{cor}
\label{b}
Any mutation-finite quiver with highest denominator $4$  is mutation-equivalent to  a symmetrization of one of the integer diagrams, i.e. either it arises from a triangulated orbifold or is one of the exceptional quivers listed in Fig.~\ref{4} (we call these {\em $F$-type quivers}).
   
\end{cor}

\begin{figure}[!h]
\psfrag{4}{\small  $\frac{1}{4}$}
\psfrag{6}{\small  $\theta_6$}
\psfrag{a}{\small  $F_4$}
\psfrag{b}{\small  $\widetilde F_4$}
\psfrag{c}{\small  $F_4^{(*,+)}$}
\psfrag{d}{\small  $F_4^{(*,*)}$}
\psfrag{e}{\small  $G_2^{(*,+)}$}
\psfrag{f}{\small  $G_2^{(*,*)}$}
\epsfig{file=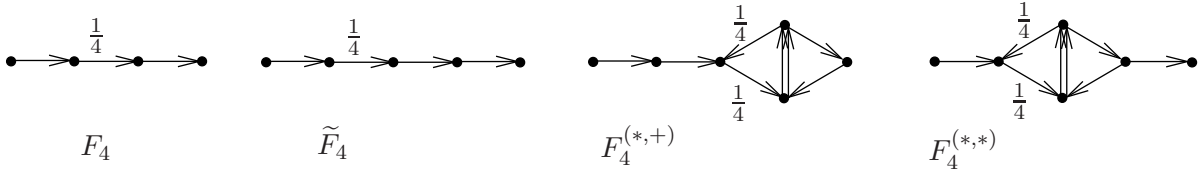,width=0.993\linewidth}
\caption{Exceptional denominator 4 quivers}
\label{4}
\end{figure}

\begin{remark}
Notice that the diagrams  $G_2^{(*,+)}$ and  $G_2^{(*,*)}$ from classification in~\cite{FeSTu2} correspond to denominator $6$ quivers and arise as elements of series shown in Fig.~\ref{ser} for $n=3$.
  
\end{remark}  

\newpage
\subsection{Denominator $5$: separating $4$ from $5$}

\begin{prop}
\label{separate 4 5}
Let $Q$ be a quiver of finite mutation type with the highest denominator $d=5$ in the mutation class.
Then no quiver in the mutation class of $Q$ contains  a denominator $4$ arrow. 

\end{prop}

\begin{proof}
%In this proof, it will be convenient to use a matrix notation, i.e. $b_{ij}=\pm\theta_{l,k}$  when there is an arrow from vertex $i$ to vertex $j$ labelled $\theta_{l,k}$. In this case we will also say that $b_{ij}$ is a {\it denominator $k$ arrow}. 

If some quiver in the mutation class of $Q$ does not contain arrows with denominators $4$, then the whole mutation class has no such arrows: this is immediate from the mutation rule (as $\sqrt 2\notin \Q(\sqrt 5)$).  
Therefore, we can assume that every quiver in the mutation class of $Q$ contains both denominator $5$ and denominator $4$ arrows.
Without loss of generality we can also assume that $Q$ is of smallest possible rank with this property.
Let $n$ be the rank of $Q$.
In view of classification of mutation-finite rank $3$ quivers we see that  $n\ge 4$.

Suppose that $Q_{n,n-1}$  is a denominator $5$ arrow. By the minimality of $Q$, no of the  arrows $Q_{i,n}$ has denominator $4$. This implies that a denominator $4$ arrow is contained in  $Q_{1,\dots,n-2}$.
Without loss of generality we can assume that the arrow $Q_{12}$ has denominator $4$.

Consider the shortest path $\mathcal P$ connecting (one of the endpoints of) $Q_{12}$ to (one of the endpoints of) 
$Q_{n-1,n}$, we can assume that $\mathcal P$ connects $Q_2$ to $Q_{n-1}$. Since $Q$ is minimal and $\mathcal P$ is shortest, the support of $\mathcal P$ coincides with $Q_{2,\dots,n-1}$ and is a linear graph containing  all vertices of $Q$ except for $Q_1$ and $Q_n$. Thus, we can assume that the subquiver   $Q_{2,\dots,n-1}$ only contains arrows $Q_{i,i+1}$, and each of these arrows is of weight $1$ or $2$.
Furthermore, besides the arrows in $Q_{2,\dots,n-1}$, denominator $4$ arrow $Q_{12}$  and denominator $5$ arrow $Q_{n,n+1}$, the quiver $Q$ may only contain two other  arrows: $Q_{13}$ and $Q_{n-2,n}$. 

Notice that $Q_{13}$ cannot have denominator $4$ as this would contradict the minimality of $Q$. Also, $Q_{13}$ cannot have weight $1$ or $2$, as in that cases the subquiver $Q_{123}$ would not be mutation-finite.
Thus, there is no arrow between vertices $Q_1$ and $Q_3$. If $Q_{23}$ has weight $2$ then $Q_{123}$ is already mutation-infinite, so we can assume that $Q_{23}$ has weight $1$.
Applying (if needed) mutation $\mu_1$ we can assume that $Q_2$ is neither a sink nor a source, so applying mutation $\mu_2$ we will create a denominator $4$ arrow $Q_{13}'$
 (and this will not affect the rest of the quiver). The subquiver spanned by all vertices but $Q_2$ will now contain both denominator $4$ and denominator $5$ arrows, which contradicts the minimality of $Q$.

% Furthermore, the minimality of $Q$ garantees that  $Q_{1,\dots,n-1}$ contains no denominator 5 arrows.
%Hence, the subquiver $Q_{1,\dots,n-1}$ is a denominator 4 subquiver. 

%In view of Corollary~\ref{b},   $Q_{1,\dots,n-1}$ is mutation equivalent to r $B_n$, or $\widetilde B_n$, or $\widetilde C_n$.
%We will assume that  $Q_{1,\dots,n-1}$ itself is such a quiver and the arrow between vertices $1$ and $2$ is a denominator 4 arrow (though this may imply that no the arrow between vertices $n$ and $n-1$ is not a denominator 5 anymore, at the same time, at least one arrow incident to vertex $n$ is still a denominator 5 arrow).
%If some of the arrows $v_nv_i$, $i<n-1$ is a denominator $5$ arrow, then $Q$ contains a connected subquiver $Q\setminus v_{n-1}$  with both denominator 4 arrow $v_1v_2$ and denominator 5 arrow  $v_nv_i$.
%So, the only denominator 5 arrow of $Q$ is $v_nv_{n-1}$. Furthermore, if at least one arrow $v_nv_i$, $i<n-1$ is present in $Q$ 
%Then $Q\setminus v_{n-1}$ is a connected subquibver containing both denominator 4 and denominator 5 arrows.
%Moreover, applying mutation at vertex 2 we  obtain a very similar quiver $Q'$ but with a denominator 4 arrow $v_1v_3$.
%So, the subquiver $Q'\subset v_2$ contains both denominator 4  and denominator 5 arrows which is impossible by assumption.
%The contradiction shows the lemma.

\end{proof}

\subsection{Denominator $5$ mutation classes}
In this section we classify denominator $5$ mutation classes (i.e. low denominator quivers containing arrows marked $\frac{1}{5}$ or $\frac{2}{5}$). In view of Proposition~\ref{separate 4 5}, such a quiver only contains arrows marked $\frac{1}{2}$ (such arrows are absent),  $\frac{1}{3}$ (simple arrows),  $\frac{1}{5}$,  $\frac{2}{5}$ and $2$ (double arrows).

The classification can be now achieved by a short (computer assisted) case-by-case study which we organize as follows. 

All rank $3$ mutation-finite  classes are listed in Theorem~\ref{finite-thm} (there are only $3$ mutation classes containing arrows of denominator $5$). 
The fourth vertex may be added to a representative of each of these $3$ mutation classes in $9^3$ ways.
Most of the obtained quivers are mutation-infinite, so this will produce $8$  mutation-finite classes listed in the left and middle columns of Fig.~\ref{5}.
Then one can add the fifth vertex to get two mutation classes of rank $5$.
Adding the sixth vertex we get exactly one mutation class, while adding one more vertex to that one does not give any new mutation-finite quivers.

We can now summarize the results of the computation described above.

\begin{theorem}
\label{thm5}
A denominator $5$ quiver of finite mutation type is mutation-equivalent to one of the quivers listed in Table~\ref{5}. 
\end{theorem}

\begin{figure}[!h]
\psfrag{15}{\small  $\frac{1}{5}$}
\psfrag{25}{\small  $\frac{2}{5}$}
\epsfig{file=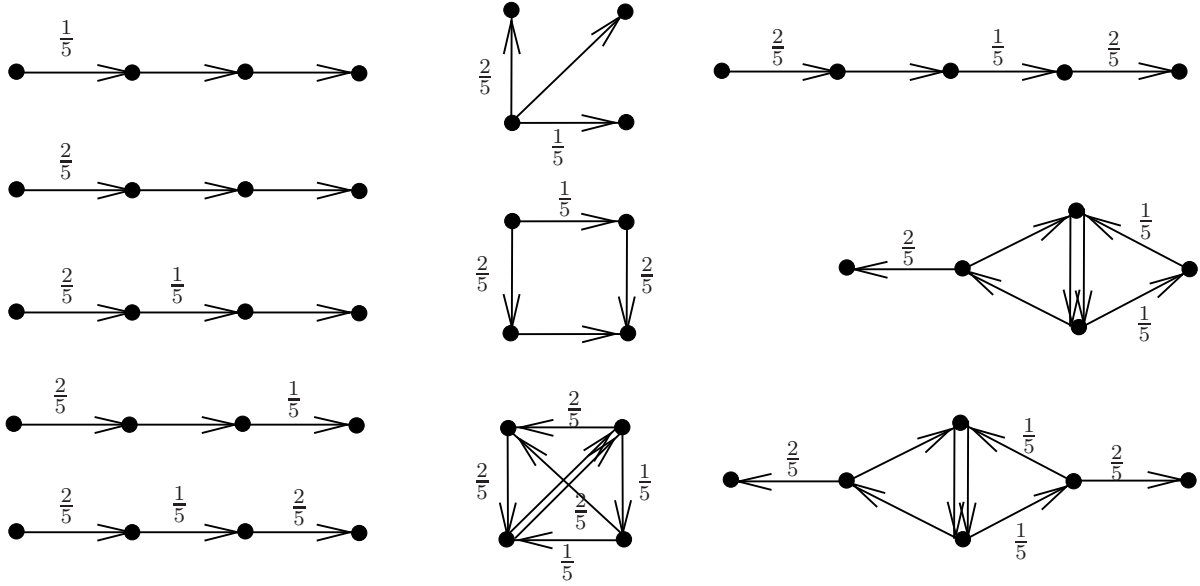,width=0.993\linewidth}
\caption{Denominator $5$ quivers of finite mutation type}
\label{5}
\end{figure}

%\begin{table}
%\label{5}
%\caption{}
%\begin{tabular}{ccccc}
%rank, Mutation class, Acyclic representatives, size of mutation class, size of labelled mutation class.
%
%\end{tabular}
%\end{table}

%For completeness we list her all finite mutation classes in rank 3: 

\section{Rank $4$ quivers with high denominators}
\label{4high}
In view of Theorems~\ref{no big} and~\ref{thm5}  we are left  to classify mutation-finite quivers of rank $4$ with  the largest denominator $d> 5$.  By Lemma~\ref{thm: ser} every such quiver is mutation-equivalent to one of the quivers shown in Fig.~\ref{ser}. In other words, we are left to study three infinite series of rank $4$ quivers. Below, we show that each mutation class in these three families is mutation-finite.

Note that all three series in Fig.~\ref{ser} are infinite (as $n\in \{2,3,\dots\}$), and computing the mutations classes for relatively small $n$ one can observe the size of the mutation classes grows with $n$. We will show by induction that all quivers in each of these mutation classes satisfy certain conditions, which will imply mutation-finiteness as the conditions describe a finite set of quivers for every given $n$. The three types of quivers shown in Fig.~\ref{ser} will be treated separately (but in a very similar way).

\begin{lemma}
\label{odd family}
The quiver shown on the right of Fig.~\ref{ser} is mutation-finite for every $n\in\{2,3,\dots \}$.

\end{lemma}

\begin{proof}
We will show by induction (on the number of mutations applied) that every quiver in the mutation class can be presented in a standard  form shown in
Fig.~\ref{standard} with some parameters $k,q,m,s\in \{0,1,2,\dots n \}$ satisfying the following conditions:

\begin{itemize}
\item[(1)] $k+q\in\{n,n+1\}$;
\item[(2)] $k>\frac{n}{2}\ge q$;
\item[(3)] $s+m+k+q=2n+1$; % and either $s+m=n$ or $s+m=n+1$;
\item[(4)] $q\le s,m \le n-q$ \ \ and \ \ $0<s,m$.

\end{itemize}
The mutation-finiteness then follows immediately.

\begin{figure}[!h]
\psfrag{11}{\scriptsize  $1$}
\psfrag{2}{\scriptsize  $2$}
\psfrag{3}{\scriptsize  $3$}
\psfrag{4}{\scriptsize  $4$}
\psfrag{k}{\normalsize  $k$}
\psfrag{q}{\normalsize $q$}
\psfrag{m}{\normalsize  $m$}
\psfrag{s}{\normalsize $s$}
\psfrag{m+q}{\normalsize  $m+q$}
\psfrag{s+q}{\normalsize $s+q$}
\quad\epsfig{file=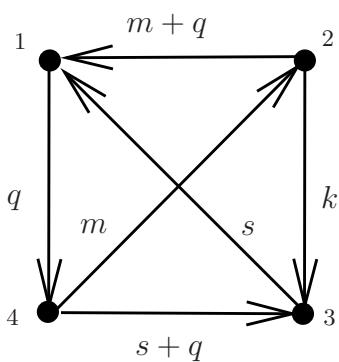,width=0.31\linewidth}
\caption{Standard form for quivers in the mutation classes shown in  Fig.~\ref{ser}. We label the arrow of weight $2\cos \frac{k\pi}{d}$ by $k$ (with $d=2n+1$ or $d=2n$ for all arrows).  }
\label{standard}
\end{figure}

\bigskip

\noindent
{\bf Base of induction.} Reordering the vertices, one can redraw the quiver shown on the right of  Fig.~\ref{ser} as in 
Fig.~\ref{standard-base3}. In this case $q=0$, $k=m=n$ and $s=1$, which clearly satisfies conditions  $(1)$--$(4)$.
\begin{figure}[!h]
\psfrag{11}{\scriptsize  $1$}
\psfrag{2}{\scriptsize  $2$}
\psfrag{3}{\scriptsize  $3$}
\psfrag{4}{\scriptsize  $4$}
\psfrag{II}{\small  $I\!I$}
\psfrag{1}{\small $\frac{1}{2n}$}
\psfrag{n-1}{\small $\frac{n-1}{2n}$}
\psfrag{1_}{\small $\frac{1}{2n+1}$}
\psfrag{n}{\small $\frac{n}{2n+1}$}
\quad\epsfig{file=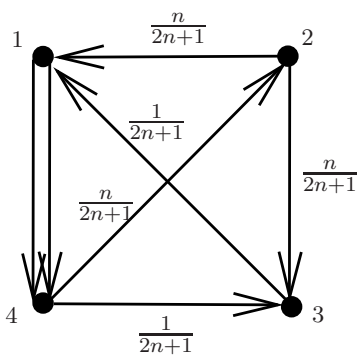,width=0.3\linewidth}
\caption{Base of induction: the quiver from the right of Fig.~\ref{ser} in the standard form.   }
\label{standard-base3}
\end{figure}

\bigskip
\noindent
{\bf Step of induction.} Our aim is now to show that the class of quivers described in Fig.~\ref{standard} with the conditions (1)--(4) is closed under mutations. {\it A priori}, we need to check four mutations for that  (one mutation in each of the four directions). However, taking into account the symmetry of the conditions above and considering the quivers up to the opposite allows us to reduce the work to checking the two mutations in the two vertices $Q_1$ and $Q_2$, see  Fig.~\ref{standard}.
Indeed, taking $Q$ to $Q\op$ and renumbering vertices according to permutation $(14)(23)$ results in the same quiver $Q$ with the label $m$ swapped with $s$ (and $m+q$ swapped with $s+q$). Now observe that taking the opposite commutes with mutations, and $Q\op$ satisfies the conditions (1)--(4) if and only if $Q$ does. Therefore, checking the mutation in, say, $Q_3$ is equivalent to checking the mutation in $Q_1$.

\medskip
\noindent
{\bf 1. Mutation in $Q_1$.} We will first check the mutation $\mu_1(Q)$. Depending on various values of $k,q,m,s$ and $n$, the quiver obtained is of one of the two forms shown in Fig.~\ref{mu_1} (in the figure we first show the mutation and then redraw the same quiver in the standard form).  In computing the new weights of arrows we apply parts (4) and (5) of  Corollary~\ref{cor3} and use the assumption  $s+m+k+q=2n+1$.  
Notice also that we obtain a weight $s-q$ (and not $q-s$) as $s\ge q$ in view of assumption (4).

\begin{figure}[!h]
\psfrag{k}{\normalsize  $k$}
\psfrag{q}{\normalsize $q$}
\psfrag{m}{\normalsize  $m$}
\psfrag{s}{\normalsize $s$}
\psfrag{m+q}{\small  $m+q$}
\psfrag{m+q.}{\scriptsize  $m+q$}
\psfrag{s+q}{\normalsize $s+q$}
\psfrag{m+2q>=n+1}{\normalsize  $m+2q\ge n+1$}
\psfrag{m+2q<=n}{\normalsize $m+2q\le n$}
\psfrag{mu1}{\normalsize  $\mu_1$}
\psfrag{s+k-q}{\small $s\!+\!k\!-\!q$}
\psfrag{s+k-q.}{\scriptsize $s\!+\!k\!-\!q$}
\psfrag{m+2q}{\small  $m+2q$}
\psfrag{m+2q.}{\scriptsize  $m+2q$}
\psfrag{s-q}{\small $s\!-\!q$}
\psfrag{s-q.}{\scriptsize $s\!-\!q$}
\epsfig{file=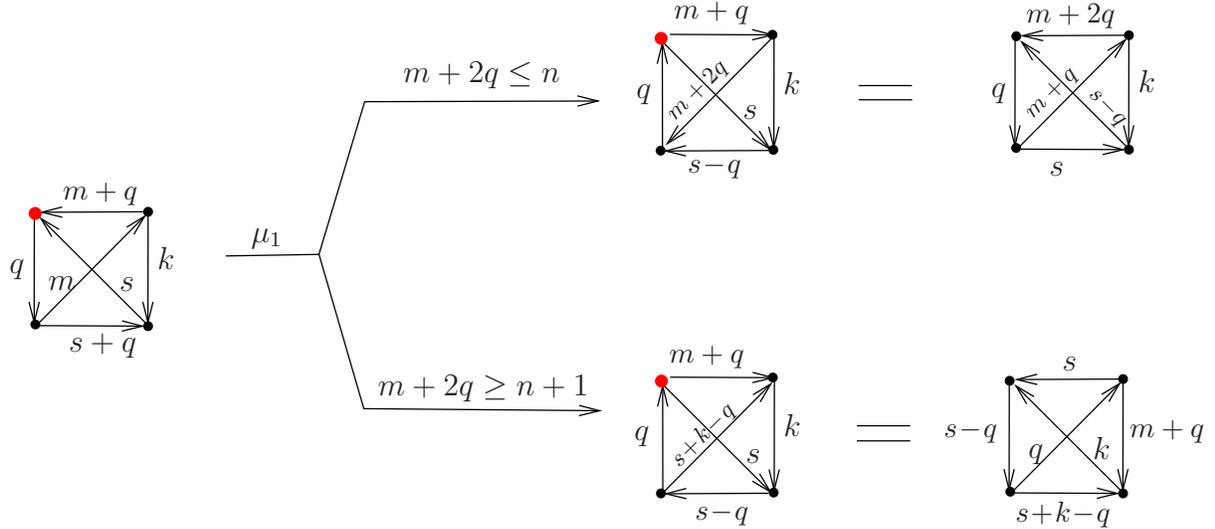,width=0.99\linewidth}
\caption{First mutation.  }
\label{mu_1}
\end{figure}

\smallskip
\noindent
{\it \underline{Case 1a}: $m+2q\le n$.} 
As follows from Fig.~\ref{mu_1}, the result of this mutation is still a quiver having the standard form shown in   Fig.~\ref{standard} with the  new values of labels 
$$
k'=k, \qquad  q'=q, \qquad s'=s-q, \qquad m'=m+q.
$$

We now need to check properties (1)--(4) for $k',q',s',m'$ (using the ones for $k,q,s,m$). We denote by (1)', (2)' etc the corresponding conditions for the mutated quiver.
 
The properties (1)'--(3)' for $k',q',s',m'$ follow immediately from the ones for  $k,q,s,m$. 

Now, we need to check (4)'. First, $s',m'>0$  (otherwise $s=q$, so $m+2q=m+s+q=2n+1-k> n$ which contradicts the assumption of the Case 1a). Next, we rewrite (4)' for $k',q',s',m'$ in terms of the old values: 
$$q\le s-q,\ m+q\le n-q$$
and prove these four inequalities. 

It is clear that $q\le m+q$ and $s-q\le n-q$. The inequality $m+q\le n-q$ also holds as  $m+2q\le n$  by the assumption of Case 1a. Finally, to prove  $q\le s-q$, assume the contrary, i.e. $s-q<q$, and hence $s<2q$. This implies $s+m<2q+m\le n$  (again, by the assumption of Case 1a), i.e. $s+m<n$. However, (1) and (3) imply that $s+m=2n+1-(k+q)=n$ or $n+1$, so we come to a  contradiction.

\smallskip
\noindent
{\it \underline{Case 1b}: $m+2q\ge n+1$.} 
The  new values of the labels are
$$
k'=m+q, \qquad  q'=s-q, \qquad s'=k, \qquad m'=q.
$$

Now, we verify properties (1)'--(4)' for  $k',q',s',m'$:

\begin{itemize} 
\item[(1)':] $k'+q'=m+s=2n+1-(k+q)$, and hence is equal to either $n$ or $n+1$. 

\item[(2)':] We need to check that $m+q>\frac{n}{2}\ge s-q$.
The first of these inequalities follows from 
$$m+q=m+2q-q\ge n+1-q\ge n+1-\frac{n}{2}=\frac{n}{2}+1,$$
while the second one follows from 
$$s-q=(s+m)-(m+q)\le n+1-(\frac{n}{2}+1)=\frac{n}{2}.  $$

\item[(3)':] $s'+m'+q'+k'=m+s+q+k=2n+1.$

\item[(4)':] First, $s',m'>0$ as $q>0$ in view of the assumption $m+2q\ge n+1$ and property (4) for $k,q,s,m$. 
Next, we  check that
$$ s-q\le k,q \le n-(s-q) $$
as follows:
\begin{itemize}
\item[($\alpha$)]
As shown in the proof of (2)', $s-q\le \frac{n}{2}$. Thus, $s-q\le\frac{n}{2}\le k$. 
\item[($\beta)$]
If $s-q>q$ then $s>2q$, which implies $m+s>m+2q\ge n+1$ in contradiction to (3). Hence, $s-q\le q$.
\item[($\gamma$)]
To show $k\le n-(s-q)$ consider two cases: $k=n-q$ and $k=n+1-q$ (one of them holds by (1)). 

If $k=n-q$ then $k+2q= n+q$ and hence $k+s\le k+2q = n+q$ (as $s\le 2q$ in view of part ($\beta$) above), which implies $k\le n-(s-q)$. 

If $k=n+1-q$ then $s< 2q$ (otherwise, $m+s\ge m+2q\ge n+1$ by the assumption of the case 1b, this would  imply  $k+q=n$  in contradiction to the assumption $k=n+1-q)$). Therefore, $k+s<k+2q= n+1+q$. This means $k<n-(s-q)+1$, and thus $k\le n-(s-q)$ as required.

\item[($\delta$)] The inequality $q\le n-(s-q)$ follows from $(\gamma)$ and $q\le k$. %As $n\ge s$, we get $q\le n-(s-q)$.

\end{itemize}
\end{itemize}

\medskip
\noindent
{\bf 2. Mutation in $Q_2$.} Now, we need to check the mutation $\mu_2(Q)$.
We follow the same scheme as before: consider two cases as shown in Fig.~\ref{mu_2} (again, we apply  Corollary~\ref{cor3} and the assumption that 
$s+m+k+q =2n+1$ to compute the new weights of some arrows).
Notice also that we obtain a weight $k-m$ (rather than  $m-k$) as 
otherwise $m>k$ would by (1) imply $m+q>k+q\ge n$, and hence $m>n-q$ which contradicts (4). 

\begin{figure}[!h]
\psfrag{k}{\normalsize  $k$}
\psfrag{q}{\normalsize $q$}
\psfrag{m}{\normalsize  $m$}
\psfrag{s}{\normalsize $s$}
\psfrag{s+q}{\small  $s+q$}
\psfrag{m+q}{\small  $m+q$}
\psfrag{m+q.}{\scriptsize  $m\!+\!q$}
\psfrag{k-m}{\normalsize $k\!-\!m$}
\psfrag{k-m.}{\scriptsize $k\!-\!m$}
\psfrag{2m+q>=n+1}{\normalsize  $2m+q\ge n+1$}
\psfrag{2m+q<=n}{\normalsize $2m+q\le n$}
\psfrag{mu2}{\normalsize  $\mu_2$}
\psfrag{s+k-q}{\small $s\!+\!k\!-\!q$}
\psfrag{s+k-q.}{\scriptsize $s\!+\!k\!-\!q$}
\psfrag{2m+q}{\small  $2m+q$}
\psfrag{m+2q.}{\scriptsize  $m+2q$}
\psfrag{s+k-m}{\small $s\!+\!k\!-\!m$}
\psfrag{s-q.}{\scriptsize $s\!-\!q$}
\epsfig{file=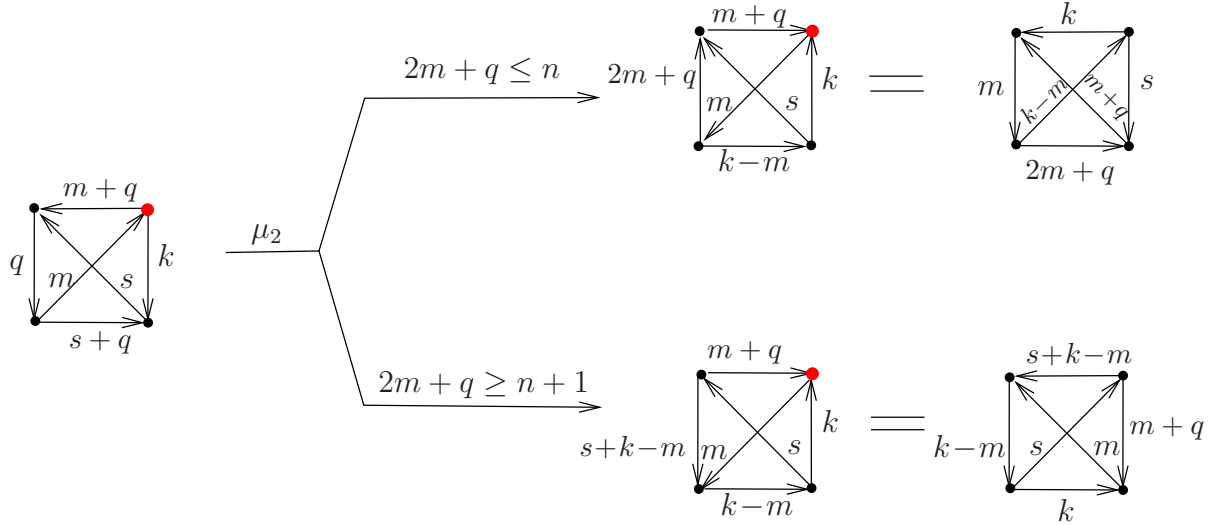,width=0.99\linewidth}
\caption{Second mutation.  }
\label{mu_2}
\end{figure}

\smallskip
\noindent
{\it \underline{Case 2a}: $2m+q\le n$.} 
The new weights of arrows in the standard form are 
$$k'=s, \qquad q'=m, \qquad s'=m+q, \qquad m'=k-m.$$

The conditions (1)'--(4)' are verified as follows:

\begin{itemize}
\item[(1)'] $k'+q'=s+m=n+1-(k+q)$ equals either $n$ or $n+1$ as it should.

\item[(2)'] We need to show $s>\frac{n}{2}\ge m$.

We start with the latter by noticing that the assumption of Case 2a implies $2m\le n-q$, and hence $2m\le n$, i.e. $m\le \frac{n}{2}$.  

Now, $s+m\ge n$ (see the proof of (1)'), so we see that $s\ge n-m\ge \frac{n}{2}$ (where the second inequality makes use of $m\le \frac{n}{2}$ shown above). If $s> \frac{n}{2}$ we are done, otherwise, $s=\frac{n}{2}$ which implies $m=\frac{n}{2}$ and $q=0$ 
(as $2m\le n-q$). Hence, $k=n$, which by (3) means $s+m=n+1$ in contradiction to $s+m=\frac{n}{2}+\frac{n}{2}=n$.

\item[(3)'] $s'+m'=k+q$ equals either $n$ or $n+1$ by (1).
\item[(4)'] The conditions $s',m'>0$ hold as $s'=m+q\ge m>0$ by (4) and $m'=k-m>0$ as $k>n/2$ by (1) and $m\le n/2$ by (2)'.
Another set of inequalities rewrites as 
$$ m\le m+q,k-m\le n-m$$
and can be proved as follows:

\begin{itemize}
\item[($\alpha$)] Clearly, $m\le m+q$.
\item[($\beta$)] By the assumption of Case 2a, $2m+q\le n$, hence, $m+q\le n-m$.
\item[($\gamma$)] $k-m\le n-m$ as $k\le n$.
\item[($\delta$)] We need to show $m\le k-m$ which is equivalent to $2m\le k$. Suppose the contrary, i.e. $k<2m$, then 
applying (1) and the assumption of Case 2a we have
$n\le k+q< 2m+q\le n$  which is impossible.
\end{itemize}

\end{itemize}

\smallskip
\noindent
{\it \underline{Case 2b}: $2m+q\ge n+1$.} 
The new weights of arrows in the standard form are 
$$k'=m+q, \qquad q'=k-m, \qquad s'=m, \qquad m'=s.$$

The computations in this case are a bit more involved than before:

\begin{itemize}
\item[(1)'] $k'+q'=k+q$, hence, it is still equal to either $n$ or $n+1$.

\item[(2)'] We need to show $m+q>\frac{n}{2}\ge k-m$. 

The first of these inequalities is obtained  (using the assumption of Case~2b) as follows:
$$m+q=\frac{2m+2q}{2}\ge \frac{2m+q}{2}\ge \frac{n+1}{2}>\frac{n}{2}. $$

To prove the second inequality, we apply  the assumption of Case~2b again and  compute
\begin{multline*}
\phantom{asdasdasd} k-m\le k-\frac{n+1-q}{2}=k+\frac{q}{2} -\frac{n+1}{2}=\\=k+q -\frac{n+1}{2}- \frac{q}{2}\le n+1 -\frac{n+1}{2}- \frac{q}{2} =\frac{n}{2}+\frac{1}{2}-\frac{q}{2}, 
\end{multline*}
which gives the required inequality if $q>0$.
%If $q>0$ then $\frac{1-q}{2}\le 0$ and $k-m\le \frac{n}{2}$ as required.

If $q=0$ then by (1) we have $k=n$. Also, the assumption of Case~2b then reads as $m\ge \frac{n+1}{2}$. Therefore,
$$k-m\le n- \frac{n+1}{2}=  \frac{n-1}{2}<  \frac{n}{2},$$  
as required.

\item[(3)'] $(m+q)+(k-m)+m+s=k+q+m+s=2n+1$ as required.
\item[(4)'] The inequalities $s',m'>0$ hold as $s'=m>0$, $m'=s>0$.
The inequalities 
$$ k-m\le m,s \le n-k+m$$
can be checked as follows:

\begin{itemize}
\item[($\alpha$)] $m\le n-k+m$ as $k\le n$.
\item[($\beta$)] $k-m\le m$ as otherwise $k>2m$  which (together with the assumption of Case~2b) implies $k+q>2m+q\ge n+1$ which contradicts (1).

\item[($\gamma$)] $k-m\le s$ as otherwise $s<k-m$, implying $s+m<k$; by (1) and (3) this means $n\le s+m<k$, i.e. $n<k$ which is impossible.  
\item[($\delta$)] To show $s\le n-k+m$ assume the contrary, i.e. $k> m+n-s$, then
$$\phantom{asdasdaaa} k+q>m+n+q-s=(2m+q)-m+n-s\ge n+1-(m+s)+n,  $$
which implies that $k+q+m+s>2n+1$ in contradiction to (3).

\end{itemize}

\end{itemize}

As no of the two  mutations takes the quiver away from the (finite) set of quivers having the standard form described by Fig.~\ref{standard} and conditions (1)--(4), we conclude that the mutation class is finite.

\end{proof}

%The mutations-finiteness of quivers in the other two families shown in  Fig.~\ref{ser} can be proved in a very similar way 

Using very similar computations we prove the following lemma.

\begin{lemma}
\label{even families}
The quivers shown on the left and in the middle  of Fig.~\ref{ser}  are mutation-finite for every $n \in\{2, 3, . . . \}$.
\end{lemma}

To prove this lemma we use exactly the same standard form of the quivers (see Fig.~\ref{standard}) together with a marginal variation of the set of conditions (see Table~\ref{conditions}). These variations  (as well as different shapes of quivers) are due to different parity of the denominator: there is one mutation class for every $n\ge 2$ with odd denominator $2n+1$,  while in the case of the even denominator $2n$ the set of quivers splits into  two mutation classes for every $n$ (see also Proposition~\ref{ser:classes}).  
%For comparison, we collect the conditions for all three families in Table~\ref{conditions}.

\begin{center}
\begin{table}
\caption{Conditions for three families}
\label{conditions}
\epsfig{file=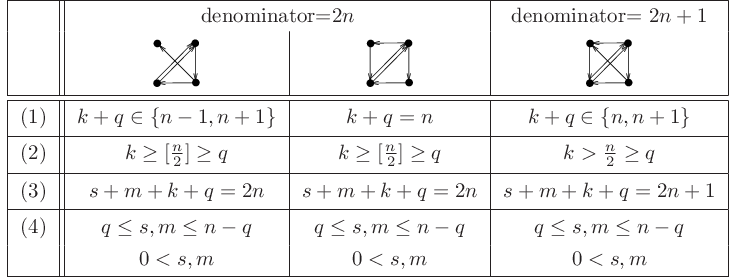,width=0.8\linewidth}
\end{table}
\end{center}

\begin{prop}
\label{ser:classes}
If $Q$ is  a quiver in the standard form (as in Fig.~\ref{standard}) satisfying the conditions as in Table~\ref{conditions},
then $Q$ is mutation-finite.
Moreover, two such quivers belong to the same mutation class if and only if they have the  same denominator and satisfy the same set of conditions.

\end{prop}

\begin{proof}
From the proof of Lemmas~\ref{odd family} and~\ref{even families} we see that 
applying mutations to any quiver $Q$ represented in the standard form and satisfying one of the three versions of the conditions we always obtain quivers of the same family. As each family is finite for any given $n$, this shows mutation-finiteness of $Q$.   

We are left to discuss which quivers belong to the same class. It is clear that quivers with different denominators (or with the same even denominator but different sets of conditions) belong to different mutation classes.
On the other hand, by  Theorem~\ref{thm: ser} every mutation-finite high denominator quiver is mutation equivalent to one of the quivers in 
 Fig.~\ref{ser}.  So, quivers with the same invariants (i.e. the same denominator and the same set of conditions) are mutation-equivalent, while quivers with different invariants are not.

\end{proof}

This concludes the proof of Theorem~\ref{intro1}.

\section{Geometric realization for finite mutation classes}
\label{sec real}
In this section we will show  that every non-integer mutation-finite mutation class (except, possibly, for ones of the orbifold type) admits a geometric realization by reflections in some positive semi-definite quadratic space $V$.
%These realizations will satisfy the following properties:
%\begin{itemize}
%\item for quivers of ``finite type'', the space $V$ will be positive-definite, i.e. reflections will act on a sphere and lie in a finite reflection group;    
%
%\item for quivers of ``affine type'', the space $V$ will be positive semi-definite of corank $1$, the reflections will act on a Euclidean space;
%
%\item for all other quivers the corank of the quadratic form will be equal to $2$, the corresponding reflection group is constructed in a way similar to a construction of extended affine Weyl groups.  
%
%\end{itemize}
%
This will 
%be proved in Theorem~\ref{realisation} and will 
allow us to define the finite, affine and extended affine types of quivers. 
%(a quiver is of finite type when it has a realization on a sphere, and is affine when it has a realization on a Euclidean space of corank $1$). 
%We refer to~\cite{FeTu} for the definitions and numerous examples of geometric realizations of quivers and their mutation classes by reflections. 

\subsection{Definitions and results}

First, we recall the necessary details from~\cite{FT-rk3,roots}.
\begin{definition}
Let $B$ be an $n\times n$ skew-symmetric matrix corresponding to a quiver $Q$, and let $V$ be a real quadratic space. 
We say that a tuple of vectors ${\bf v}=(v_1,\dots,v_n)$ is a {\it geometric realization} of $Q$ if the following conditions hold:
\begin{itemize}
\item[(1)] $(v_i,v_i)=2$ for $i=1,\dots,n$, $|(v_i,v_j)|=|b_{ij}|$ for $1\le i<j\le n$;
\item[(2)] if $Q_{i_1,i_2,i_3}$ is a cycle, then the number of pairs $(j,k)$ such that  $(v_{i_j},v_{i_k})>0$ is even if $Q_{i_1,i_2,i_3}$ is acyclic and odd if $Q_{i_1,i_2,i_3}$ is cyclic. 
\end{itemize}
A {\em mutation} $\mu_k$ of ${\bf v}$ is defined by partial reflection:
$$
\mu_k(v_j)=\begin{cases}  
v_j-(v_j,v_k)v_k & \text{if  $b_{jk}>0$,}\\ 
-v_k & \text{if $j=k$,}\\
v_j&  \text{otherwise.}\end{cases}
$$
We say that $\bf v$ provides a {\it realization by reflections} of the mutation class of $Q$ if the mutations of 
$\bf v$ agree with the mutations of $Q$, i.e. if conditions $(1)$--$(2)$ are satisfied after every sequence of mutations.
\end{definition}
We recall that every acyclic integer quiver admits a realization by reflections~\cite{S2,ST}.  
Following~\cite{S}, we give the following definition.

\begin{definition}
A geometric realization of a quiver $Q$ by vectors ${\bf v}=\{v_1,\dots,v_n\}$ is {\it admissible} if for every chordless oriented cycle  $Q_{i_1},\dots, Q_{i_s}$ of $Q$ the number of positive scalar products $(v_{i_j},v_{i_{j+1}})$ is odd,
while in every chordless non-oriented cycle such a number is even. A geometric realization of a mutation class is admissible if the realization of every quiver is admissible.

\end{definition}

We will start by showing that every non-orbifold finite mutation class of non-integer quivers has a representative possessing an admissible geometric realization.

\begin{lemma}
\label{initial}
Every quiver shown in Table~\ref{table mut-fin} has an admissible geometric realization.  

\end{lemma}  

\begin{proof}
Every quiver listed in Table~\ref{table mut-fin} is of one of the following three types: 
\begin{itemize}
\item[-] either the rank is $3$ (and then it is $\t G_{2,n}$ or $H_3$);% which is mutation-acyclic and thus has a realization by~\cite{FT-rk3};
\item[-] or it is acyclic;
\item[-] or it has a double arrow, and by removing either end of the double arrow we obtain an acyclic quiver (the two acyclic quivers are the same up to one source/sink mutation). 
\end{itemize}

Quiver  $\t G_{2,n}$ is mutation-acyclic of rank $3$, and thus has an admissible realization by~\cite{FT-rk3}.
 
For an acyclic quiver $Q$, we define inner product on vectors $v_1,\dots,v_n$ by $(v_i,v_i)=2$, $(v_i,v_j)=-w_{ij}$, where $w_{ij}$ is the weight of the arrow $Q_{ij}$. Clearly, ${\bf v}=\{v_1,\dots, v_n\}$ is an admissible realization of $Q$. 

For the last type of quivers, assume that the ends of the double arrow are $Q_1$ and $Q_n$. We take the acyclic subquiver $Q'$ obtained by removing $Q_n$, define inner product on vectors $v_1,\dots,v_{n-1}$ for the acyclic subquiver $Q'$ as described above, and then define $(v_n,v_n)=2$, $(v_n,v_i)=(v_1,v_i)$ for $i<n$. Then  ${\bf v}=\{v_1,\dots, v_n\}$ is an admissible realization of $Q$. 
%Notice that all the matrices are either the rank of the space (of the matrix?) is as needed ...
\end{proof}  

\begin{remark}
\label{rem x5}  
The condition that $Q$ is not of orbifold type is necessary for Lemma~\ref{initial}: it is easy to check that
already  the  surface quiver shown in Fig.~\ref{x5}(a) has no admissible geometric realization.
%, see Remark~\ref{pf of killhope}.
This quiver defines a triangulation of a once punctured annulus, another representative of the same mutation class is shown in Fig.~\ref{x5}(b).

%See also~\cite[]{S} for another example of a surface quiver without admissible realization. 

%In the class of surface quiver such a quiver is unique ...

%Is there any orbifold like this? Are there many?
  
\end{remark}

%\begin{figure}[!h]
%\epsfig{file=./pic/killhope.eps,width=0.25\linewidth}
%\caption{A surface quiver having no admissible geometric realization. }
%\label{killhope}
%\end{figure}

\begin{figure}[!h]
\psfrag{a}{\small (a)}
\psfrag{b}{\small (b)}
\epsfig{file=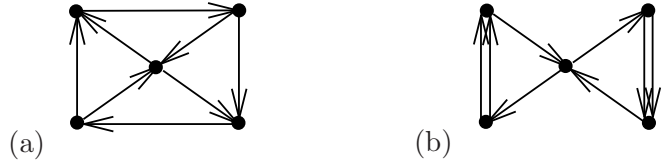,width=0.55\linewidth}
\caption{The surface quiver (a) has no admissible geometric realization.}
\label{x5}
\end{figure}

%\begin{remark}
%\label{initial v}
%Each of the quivers  Fig.~\ref{ser} have a geometric realisation. To construct it, we take the acylic rank 3 subquiver $Q_{124}$, obtain a %geometric realisation $\{v_1,v_2,v_4\}$ in $\E^2$ (this exists by~\cite{FT-rk3}), then we take another copy of $v_1$ for $v_3$.
%Clearly,  vectors  $\bf v=\{v_1,v_2,v_3,v_4\}$, where $v_3=v_1$ provide geometric realisation for $Q$.
%In addition,  by ???  in ~\cite{FT-rk3} we see that this geometric realisation is admissible.
%\end{remark}  

\begin{theorem}
\label{realisation}  
Let $Q$ be a  real  %non-integer
mutation-finite quiver of rank higher than $3$ not originating from an orbifold. Then the mutation class of $Q$ admits a geometric realization by reflections in a positive or semi-positive quadratic space $V$. In particular, the quadratic form has the kernel of dimension 
\begin{itemize}
\item[-] zero for quivers in the top row of Table~\ref{table mut-fin};
\item[-] one for quivers in the middle row of  Table~\ref{table mut-fin}; 
\item[-] two for quivers in the bottom row of  Table~\ref{table mut-fin}.
\end{itemize}
\end{theorem}

\begin{proof}
In Lemma~\ref{initial} we have constructed geometric realizations for representatives of required mutation classes, so we only need to show that these geometric realizations can be extended to the whole mutation classes.
For rank $3$ mutation classes we know the result from~\cite{FT-rk3}.  
For the three series in rank $4$ this will be done in Section~\ref{geom for ser}. Other mutation classes  are treated case-by-case.

The case-by-case check is done via a code which verifies that the realization of the initial quiver propagates as a realization of the whole mutation class. The algorithm is the following: we apply a mutation to a quiver and the  partial reflections to the corresponding set of vectors (i.e., mutate the Gram matrix according to the rules prescribed in~\cite{BGZ}), and then verify that the mutated Gram matrix provides an admissible realization of the mutated quiver. Notice that the mutated Gram matrix only  depends on the initial Gram matrix and the directions of arrows in the corresponding quiver before the mutation, but does not depend on the actual vectors $v_1,\dots, v_n$. The code checks that in each of the (non-serial) finite mutation classes the pair (Gram matrix; exchange matrix) takes only finitely many values and the entries in the Gram matrix and the exchange matrix agree, i.e. $|(v_i,v_j)|=|b_{ij}|$ for $i\ne j$.

The dimension of the kernel can be easily seen from the initial construction in Lemma~\ref{initial}.   
%For each of the rank $4$ non-serial mutation-finite quivers we have also checked the statement by hands.
\end{proof}

\begin{remark}
\label{unpunctured}
It follows from Remark~\ref{rem x5} that already the mutation class of a quiver originating from a punctured annulus (see Fig.~\ref{x5}) does not have an admissible realization by reflections.
% This implies that any non-acyclic mutation class of a punctured surface or orbifold does not possess an admissible realization by reflections.
This implies that most non-acyclic mutation classes of punctured surfaces or orbifolds do not possess admissible realizations by reflections.

Geometric realizations by reflections of all mutation classes of quivers originating from unpuctured surfaces were constructed in~\cite{FeLSTu}. There is a strong evidence for the following conjecture. 
\end{remark}

\begin{conjecture}
Every mutation class originating from an unpunctured orbifold admits a realization by reflections.

\end{conjecture}

\subsection{Geometric realizations for rank $4$ series}
\label{geom for ser}

In rank $4$ we have infinitely many finite mutation classes whose sizes are not uniformly bounded, so we are not able to apply a computer verification. We start with proving the following technical lemma.

\begin{lemma}
\label{empty arrow}
Let $Q$ be a quiver mutation-equivalent to one of the quivers in Fig.~\ref{ser}.
If $Q$ has a vanishing arrow then $Q$ is one of the quivers shown in Fig.~\ref{f_empty_arrows}.

\end{lemma}  

\begin{proof}
$Q$ can have a vanishing arrow in the only case when the highest denominator of $Q$ is even, and hence $Q$ is mutation-equivalent to the quiver on the left or in the middle of  Fig.~\ref{ser}. For each of these quivers (considered in the standard form) we check  which arrows can vanish  (we use the conditions shown in Table~\ref{conditions} for that; an arrow marked $x/2n$ vanishes if and only if $x=n$). In particular, condition (2) implies that $q\ne n$.

Further, if $k=n$ then the conditions imply that no other arrow vanish, and moreover, the quiver is as on Fig.~\ref{f_empty_arrows}(a) or (e). If $k\ne n$ we check the case $m=n$ and get the quiver on   Fig.~\ref{f_empty_arrows}(b)  (there are no such quivers in the other mutations class). The same (up to symmetry) will happen if $s=n$. Finally, if  $k,m,s\ne n$ and $m+q=n$ we obtain the quivers in  Fig.~\ref{f_empty_arrows}(c) and (f). If, in addition, we require $s+q=n$, we get the quivers shown in Fig.~\ref{f_empty_arrows}(d) and (g). 
  
\end{proof}

\begin{figure}[!h]
\psfrag{a}{\small (a)}
\psfrag{b}{\small (b)}
\psfrag{c}{\small (c)}
\psfrag{d}{\small (d)}
\psfrag{e}{\small (e)}
\psfrag{f}{\small (f)}
\psfrag{g}{\small (g)}
\psfrag{s<m}{\scriptsize $s<m$}
\psfrag{s+m=n+1}{\scriptsize $s+m=n+1$}
\psfrag{s+m=n}{\scriptsize $s+m=n$}
\psfrag{s+m=n-1}{\scriptsize $s+m=n-1$}
\psfrag{n-1/2}{\scriptsize  $\frac{n-1}{2}$}
\psfrag{n+1/2}{\scriptsize  $\frac{n+1}{2}$}
\psfrag{n/2}{\scriptsize  $\frac{n}{2}$}
\psfrag{n-m+s}{\scriptsize  $n-m+s$}
\psfrag{n-m}{\scriptsize  $n-m$}
\psfrag{n-s}{\scriptsize $n-s$}
\psfrag{1}{\scriptsize  $1$}
\psfrag{n-1}{\scriptsize  $n-1$}
\psfrag{2s}{\scriptsize $2s$}
\psfrag{m}{\scriptsize  $m$}
\psfrag{s}{\scriptsize $s$}
\psfrag{m+1}{\scriptsize  $m+1$}
\psfrag{s+1}{\scriptsize $s+1$}
\epsfig{file=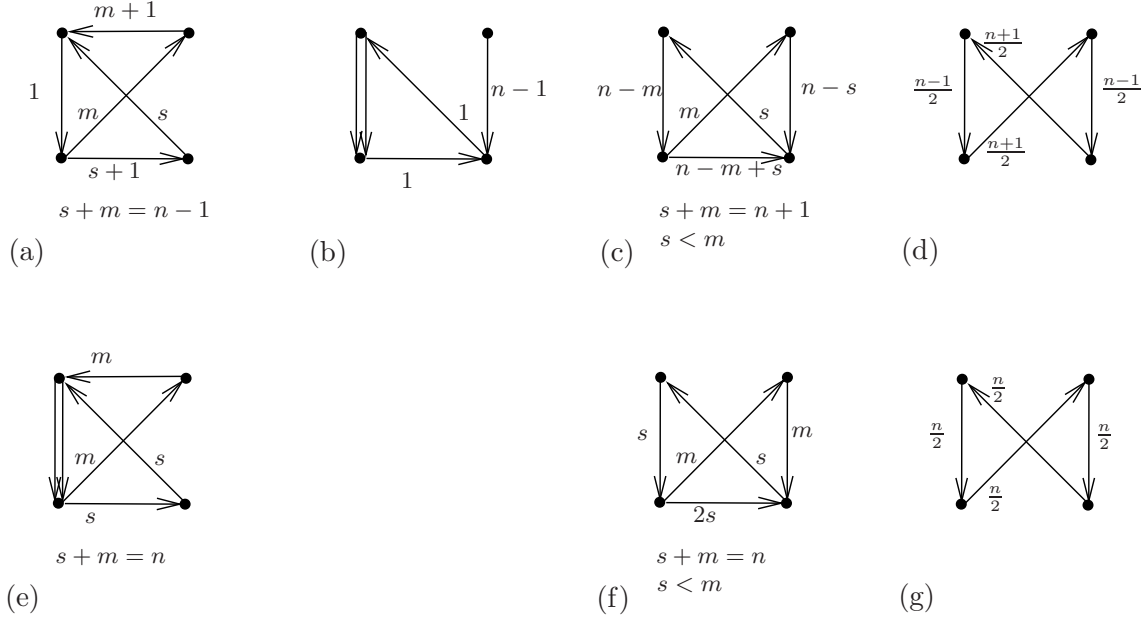,width=0.99\linewidth}
\caption{Quivers with vanishing arrows: (a)--(d) and (e)--(g) belong to two different mutation classes respectively, cf. Table~\ref{conditions}.   }
\label{f_empty_arrows}
\end{figure}

\begin{lemma} 
\label{series-geom}
Let $Q$ be a quiver in its standard form (see Fig.~\ref{standard}) and $\mathbf v\!=\!\{v_1,\dots,v_4\}\!\subset\! V$ its admissible geometric realization.
Then for every $i\in\{1,\dots,4 \}$ the collections of vectors  $ \mu_i(\mathbf v)$ provide geometric realizations of $\mu_i(Q)$.   

\end{lemma}

\begin{proof}
  
The proof is by induction on the number of mutations needed to reach a given quiver $Q$ from the initial quiver $\hat Q$ shown in Fig.~\ref{ser}.
We start with a quiver shown in  Fig.~\ref{ser} and consider its admissible geometric realization $\hat{\bf  v}$ constructed in Lemma~\ref{initial}.
Given a quiver $Q$ in the mutation class and its admissible realization $\bf v$,
we will apply all four possible mutations (mutating the set of vectors using partial reflections)  and check that the mutated set of vectors $\bf v'=\mu(\bf v)$ provides an admissible geometric realization for the mutated quiver $Q'=\mu(Q)$ (note that, as in Lemma~\ref{odd family}, we actually need to check only two mutations, the other two follow from a symmetry of the quiver provided by taking $Q\op$ with a permutation of vertices). 
Since $\bf v$ is an admissible realization for $Q$, we conclude that $\bf v'$ is a geometric realization for $Q'$ (see~\cite[Proposition 3.2]{BGZ}),
and we only need to show that  $\bf v'$ is an {\it admissible} geometric realization for $Q'$.

We start checking the admissibility of $\bf v'$  by considering the case of odd denominator: this will be the easiest case as no quiver in the mutation class has vanishing arrows.

\medskip
\noindent
{\bf Case 1: odd denominators.}
We label an arrow $Q_{ij}$  of $Q$ by ``$+$''  (resp. ``$-$'') if $(v_i,v_j)$ is positive (resp. negative). 
%By ????  in~\cite{FT-rk3} we see that  acyclic (resp. cyclic) rank $3$ subquivers of $Q$ contain  even (resp. odd) number of arrows labeled with   ``$+$''.
%Notice, that this property is preserved  when we change the sign of one of $v_i$. Moreover, up to changing signes of some of $v_i$
%(with according change of labels of the arrows incident to $v_i$, there is a unique way to label the arrows of $Q$ (having a standard form shown in  Fig.~\ref{standard}) with ``$+$'' a ``$-$''
%so that acyclic (resp. cyclic) rank 3 subquivers of $Q$ contain  even (resp. odd) number of arrows labeled with   ``$+$''.

Applying a mutation $\mu_i$ we compute the new sign labels as follows. First, we compute the new labels of all arrows $Q_{ij}'$ incident to $Q_i$: these labels easily follow from the mutation rule (which says  that  either both vectors are reflected in $v_i$ and, hence, the sign is preserved, or only $v_i$ is substituted by its negative, and then the sign changes to the opposite). The label for an arrow $Q_{lj}'$ non-incident to $i$ is computed from the triangle $Q_{ijl}'$: namely, the number of arrows labeled  by ``$+$'' in  $Q_{ijl}'$ should be even if and only if $Q_{ijl}'$ is cyclic by~\cite{FT-rk3}. When all labels are computed, we check the rank $3$ subquiver $Q'\setminus Q_i'$ and see that the labeling is also admissible on this rank $3$ subquiver, see Fig.~\ref{signs}. This implies that $\bf v' $ is an admissible realization for $Q'$ (indeed, in the assumption of odd denominator we only need to check cycles of length $3$).
Notice that as all quivers in the mutation class are ones in the standard form and no arrow vanishes from it, mutations considered in 
 Fig.~\ref{signs} exhaust all possibilities for the case of odd denominator (here, we use the two possible forms of mutated quiver explored in the proof of Lemma~\ref{odd family}, see Figs.~\ref{mu_1} and~\ref{mu_2}).%, and also the symmetry of the quiver by taking $Q\op$ with a permutation of vertices).

\begin{figure}[!h]
\psfrag{-}{\color{red}   $\mathbf -$}
\psfrag{+}{\color{red}   $\mathbf +$}
\psfrag{1}{\scriptsize $1$}
\psfrag{2}{\scriptsize $2$}
\psfrag{3}{\scriptsize $3$}
\psfrag{4}{\scriptsize $4$}
\psfrag{mu1}{\small $\mu_1$}
\psfrag{mu2}{\small $\mu_2$}
\epsfig{file=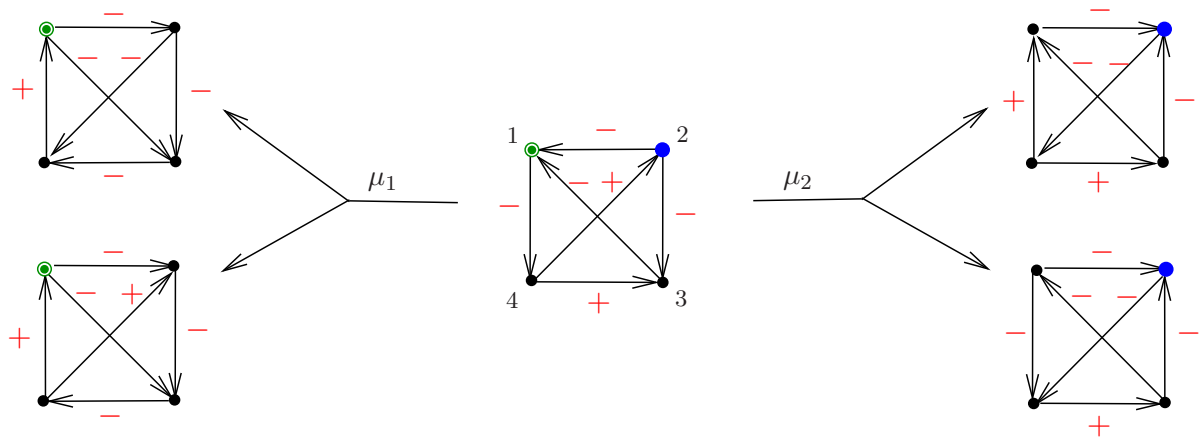,width=0.99\linewidth}
\caption{Mutating signs for odd denominators. For each mutation we consider two cases (depending on the weights of arrows, see Figs.~\ref{mu_1} and~\ref{mu_2}). }
\label{signs}
\end{figure}

\medskip
\noindent
{\bf Case 2: even denominators.}
We follow exactly the same plan as for odd denominator, however, we need to consider the quivers with  some vanishing arrows separately (as in this case we need to take additional care while mutating the sign labels). 

In Lemma~\ref{empty arrow} we  list the quivers with vanishing arrows appearing in the considered mutation classes.
Given a quiver $Q$ from one of the two series, we need to do the following:
\begin{itemize}
\item[(1)] if $Q$ has no vanishing arrows and ${\bf v}$ is an admissible geometric realization of $Q$, then we need to check the admissibility of the realization $\mu({\bf v})$;
\item[(2)] if $Q$ has vanishing arrows, $\bf v$ is an admissible geometric realization of $Q$ and  $Q'=\mu(Q)$, then we need to check whether ${\bf v'}=\mu(\bf v)$ is a geometric realization of $Q'$, and whether it is admissible.  
\end{itemize}
In the first of these checks the condition for cycles of length $3$ is verified by the same computation as before, however, we need to check also cycles of length $4$ (as $\mu(Q)$ may have vanishing arrows). As we can see from the list in Fig.~\ref{f_empty_arrows}, a length $4$ chordless cycle is always oriented (for quivers in these mutation classes). As one can check in Fig.~\ref{signs}, an oriented cycle of length $4$ always receives an odd number of labels ``$+$'', even when this cycle is not a chordless one. This verifies the condition for cycles of length $4$, and hence we can assume that $Q$ has at least one vanishing arrow. 

To complete the second check above, we need to mutate the quiver. As before, we label the arrows of $Q$ with ``$+$'' and ``$-$'' in an admissible way (which exists by the induction assumption). Note that for each of the quivers in the list there is a unique way to choose such a labeling (up to changing signs of some of vectors $v_1,\dots,v_4\in \bf v$). % with appropriate change of labels  ``$+$'' and ``$-$''). 
We will only need to look at the mutations in the directions of vertices incident to some vanishing arrows (all other mutations are treated in exactly the same way as before).
Also, we do not need to check the mutations with respect to sink or source as they do not change signs in any oriented cycle and change exactly two signs in a non-oriented one. Furthermore, we will use symmetries of quivers (and the symmetry up to taking $Q\op$) to reduce the computations.
% (in particular, as exact values of weights of arrows do not play any role in our computation, we only check one mutation for the quiver shown in Fig ....). 
This reduces the list of cases to the one in Fig.~\ref{signs-even}.

To compute the signs after mutation $\mu_i$ we do the following. First, we compute the signs of all arrows incident to $Q_i$ as before.
Then, we compute the signs of all arrows $Q'_{jl}$ where both $Q_j$ and $Q_l$ are connected to $Q_i$ by a non-vanishing arrow.
All the other signs remain intact (indeed, if both $Q_j$ and $Q_l$ are not joined to $Q_i$ then $v_j=v'_j$ and $v_l=v'_l$;
if the arrow $Q_{ij}$ vanishes and $Q_{il}$ does not, then  $v_j=v'_j$, while $v_l'$ may either coincide with $v_i$ or computes as $v_l'=v_l-(v_i,v_l)v_i$,  in both cases we have $(v_j',v_l')=(v_j,v_l)$).
The computation shows that $\bf v'$ is an admissible realization of $Q'$. 

\end{proof}  

\begin{figure}[!h]
\psfrag{-}{\color{red} \tiny  $\mathbb -$}
\psfrag{+}{\color{red} \tiny  $\mathbb +$}
\psfrag{1}{\scriptsize $1$}
\psfrag{2}{\scriptsize $2$}
\psfrag{3}{\scriptsize $3$}
\psfrag{4}{\scriptsize $4$}
\psfrag{mu1}{\small $\mu_1$}
\psfrag{mu2}{\small $\mu_2$}
\epsfig{file=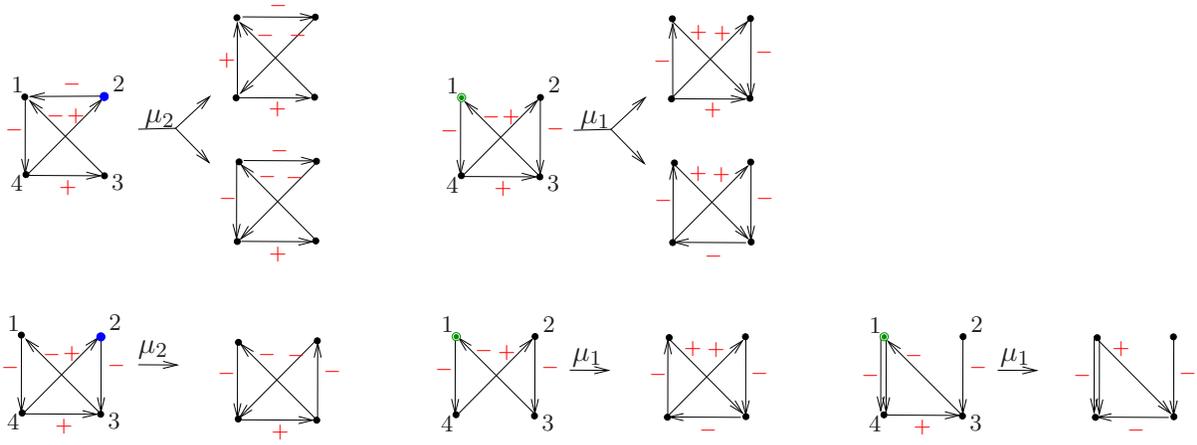,width=0.99\linewidth}
\caption{Mutating signs: additional mutations for even denominators. }
\label{signs-even}
\end{figure}

\begin{remark}
\label{seeds}
Once we have geometric realization of a mutation class by reflections, we can consider the set of all mirrors of reflections obtained by iterated mutations of the initial tuple of vectors. Then it is easy to see that a quiver is of finite type if and only if the corresponding set of mirrors is finite (and coincides with the hyperplane arrangement associated to a finite Coxeter group).   

\end{remark}

%\begin{remark}
%\label{pf of killhope}
%We can use the technique of labelling arrows of the quiver to show that the quiver in~\ref{killhope} has no admissible geometric realisation. Indeed, it is easy to show that this quiver has no admissible labelling with ``$+$'' and ``$-$''. 
%\end{remark}  

\section{Acyclic quivers  and acute-angled simplices }
\label{sec acyclic}
For non-integer quivers of finite or affine type, geometric realization constructed in Lemma~\ref{initial} defines an acute-angled simplex bounded by the mirrors of reflections of the corresponding finite or affine Coxeter group ($B_n$, $F_4$,  $H_{2}$, $H_3$, $H_4$ or $G_2^{n}$ and affine versions of them). 
It is natural to ask the following two questions:

\begin{itemize}
\item[(1)] Given a mutation-finite acyclic quiver, does it always correspond to an acute-angled simplex defined by  the geometric realization (up to a change of signs of some of the vectors $v_i$)? 

\item[(2)] Given an acute-angled simplex defined by some roots of a root system in $V$, does it give rise to a realization of a mutation-finite quiver?   
  
\end{itemize}

Notice that in rank $3$ the answers to both of these questions are positive, see~\cite{FT-rk3}. Moreover, for integer quivers this also holds: in finite (respectively, affine) types $A$, $B$, $C$, $D$, $E$, $F$ and  $G$ every acyclic quiver defines an acute-angled spherical (respectively, Euclidean) Coxeter simplex, and any acyclic orientation of a Coxeter diagram of a Coxeter simplex gives rise to a mutation-finite quiver. 

We will now see that the situation in the general case is more involved.

\subsection{Acute-angled simplices for all acyclic representatives}
The answer to the first question is positive also in the general non-integer case. We have checked it case-by-case (but we have no conceptual proof at the moment).  In Table~\ref{t acyclic} we list all acyclic quivers (up to sink-source mutations) in the mutation classes containing more than one acyclic representative.

%\begin{remark}
%One can try to find a more conceptual proof of the positive answer to the first question as follows. We know from~\cite{ST} and Theorem~\ref{realisation} above that every mutation-finite acyclic quiver $Q$ (integer or not) admits a geometric realization. This geometric realization will define an admissible sign labeling of arrows of $Q$ (as in the proof above). If $Q$ is a complete quiver (i.e. when there are no vanishing arrows), it is easy to check that there is a unique admissible labeling up to applying a sequence of simultaneous changes of signs at all arrows incident to a given vertex (which corresponds to a change of sign of the corresponding vector). Moreover, this unique labeling is equivalent to the one where all arrows are labeled with ``$-$'', which means that the corresponding geometric configuration of hyperplanes is acute-angled. For the case of non-complete quivers, in many situations one can apply similar reasoning (however, one needs to perform a case-by-case study to investigate whether this reasoning is always applicable).  
    
%\end{remark}

\begin{center}
\begin{table}
\caption{Acyclic quivers in finite mutation classes containing more than one acyclic representative (up to sink/source mutations)}
\label{t acyclic}
\epsfig{file=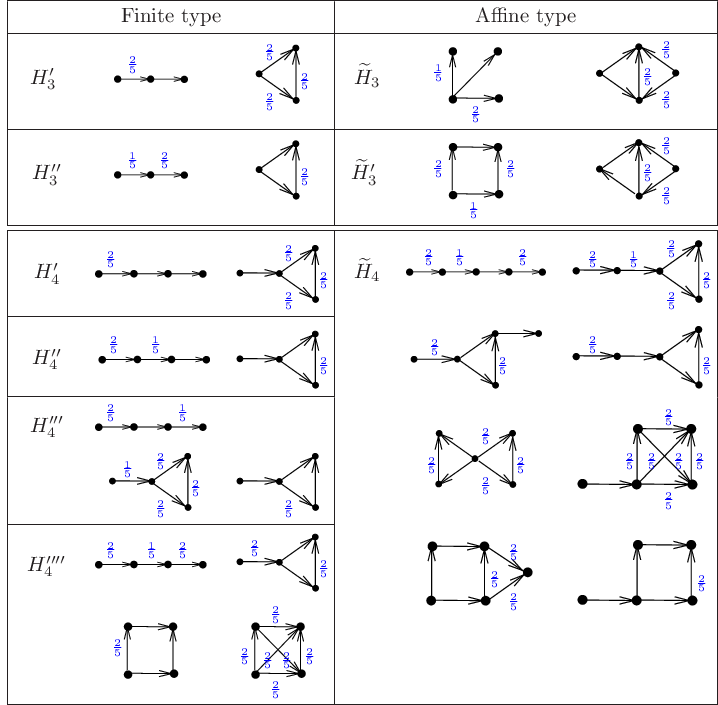,width=0.75\linewidth}  
\end{table}  
\end{center}

\subsection{Not all acute-angled simplices give rise to mutation-finite acyclic quivers}

It turns out that the answer to the second question is negative. 

By a {\em diagram} of a simplex we will mean a counterpart of a Coxeter diagram, i.e. a weighted graph, where vertices correspond to the facets of the simplex, and the weights $m/d$ of the edges denote the dihedral angles $\pi m/d$ (the edges with label $1/2$ are omitted, the edges corresponding to $\pi/3$ are unlabeled). We have written out the complete list of diagrams of acute-angled simplices in root systems
$\widetilde H_3$, $H_4$ and $\widetilde H_4$ and checked that most of these simplices appear as geometric realizations of some mutation-finite acyclic mutation classes. However, there is a number of exceptions: in Table~\ref{acute} we list all (up to sink/source mutations) mutation-infinite acyclic quivers 
appearing as orientations of diagrams of acute-angled simplices in $\widetilde H_3$, $H_4$ and $\widetilde H_4$. 

It is currently not clear to us what distinguishes the acute-angled simplices appearing in Tables~\ref{acute} from  ones defining mutation-finite quivers, and we think it would be an interesting question to understand the source of this difference.

\begin{center}
\begin{table}
\caption{Mutation-infinite acyclic quivers from acute-angled $H$-simplices (up to sink/source mutations)}
\label{acute}
\epsfig{file=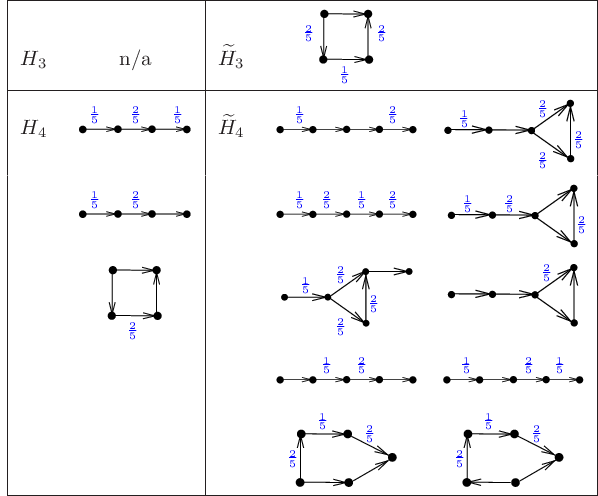,width=0.7\linewidth}
\end{table}
\end{center}

\begin{remark} Finally, we list some observations concerning the acute-angled simplices and the corresponding quivers. 
\begin{itemize}
\item[(a)] Every acute-angled simplex in  $H_3$, $\widetilde H_3$, $H_4$ or $\widetilde H_4$
either is decomposable (i.e., its diagram is disconnected), or is a spherical Coxeter simplex of the type $H_3$ or $H_4$, or has a diagram whose orientation appears either in
Table~\ref{t acyclic} or in Table~\ref{acute} (or in both: two distinct acyclic orientations of the same simplex diagram may not be simultaneously mutation finite/infinite). %This is clear since every acyclic orientation of a diagram build from such a simplex  gives either a mutation-finite or a mutation-infinite mutation class (and then either coincides with $H_3$ or $H_4$, or appers in Table~\ref{t acyclic} or  Table~\ref{acute} respectively).

\item[(b)] Notice that when a diagram of a simplex has a cycle of length more than $3$, there are two acyclic orientations of such a diagram up to sink/source mutations. All other diagrams arising from acute-angled simplices have a unique acyclic orientation up to sink/source mutations.

%\item[(c)] There are two (non-oriented) cycles of length $4$ arising from acute-angled $H$-simplices: one in $\widetilde H_3$ and one in $H_4$. Both acyclic orientations of the cycle arising from $H_4$ result in mutation-infinite quivers. For the cycle arising from  $\widetilde H_3$, one orientation gives rise to a mutation-finite quiver $\widetilde H_3'$, while the other one produces a mutation-infinite quiver.

%\item[(c)] Both acyclic orientations of the cycle of length $5$ are mutation-infinite.

\item[(c)] Some of the $\widetilde H_4$-quivers in   Table~\ref{acute} are mutation-equivalent.
  In particular, this is the case for the two quivers shown in the first, second and the third rows
  (see Fig.~\ref{mut} for the sequences of mutations).

\item[(f)] Acute-angled simplices in finite types are also listed in~\cite{F}.

\item[(g)] The affine extensions of Coxeter groups of type $H$ were described in~\cite{PT}. In particular, the diagram of the simplex giving rise to the top left $\widetilde H_4$ quiver in Table~\ref{acute} was used to define the group $\widetilde H_4$. We note that one can start with any of the simplices whose diagrams are listed in the $\t H_4$ parts of Tables~\ref{t acyclic} and~\ref{acute} to get the same group.

\end{itemize}
  
\end{remark}

\begin{figure}[!h]
  \epsfig{file=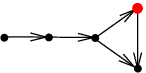,width=0.14\linewidth}\put(-19,25){\tiny\color{blue} $\frac{2}{5}$}\put(-19,0){\tiny\color{blue} $\frac{2}{5}$}\put(-2,10){\tiny\color{blue} $\frac{2}{5}$}\put(-57,20){\tiny\color{blue} $\frac{1}{5}$}
  \qquad
  \raisebox{15pt}{\epsfig{file=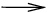,width=0.05\linewidth}}
  \qquad
  \epsfig{file=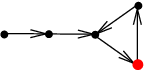,width=0.14\linewidth}\put(-19,25){\tiny\color{blue} $\frac{2}{5}$}\put(-2,10){\tiny\color{blue} $\frac{2}{5}$}\put(-57,20){\tiny\color{blue} $\frac{1}{5}$}
  \qquad
    \raisebox{15pt}{\epsfig{file=./pic/arrow.eps,width=0.05\linewidth}}
  \qquad
  \epsfig{file=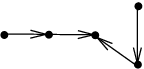,width=0.14\linewidth}\put(-2,10){\tiny\color{blue} $\frac{2}{5}$}\put(-57,20){\tiny\color{blue} $\frac{1}{5}$}

\vspace{15pt}  

  \epsfig{file=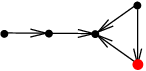,width=0.14\linewidth}\put(-2,10){\tiny\color{blue} $\frac{2}{5}$}\put(-57,20){\tiny\color{blue} $\frac{1}{5}$}\put(-38,20){\tiny\color{blue} $\frac{2}{5}$}
  \qquad
  \raisebox{15pt}{\epsfig{file=./pic/arrow.eps,width=0.05\linewidth}}
  \qquad
  \epsfig{file=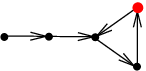,width=0.14\linewidth}\put(-19,25){\tiny\color{blue} $\frac{1}{5}$}\put(-2,10){\tiny\color{blue} $\frac{2}{5}$}\put(-57,20){\tiny\color{blue} $\frac{1}{5}$}\put(-38,20){\tiny\color{blue} $\frac{2}{5}$}
  \qquad
    \raisebox{15pt}{\epsfig{file=./pic/arrow.eps,width=0.05\linewidth}}
  \qquad
  \epsfig{file=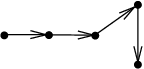,width=0.14\linewidth}\put(-2,10){\tiny\color{blue} $\frac{2}{5}$}\put(-57,20){\tiny\color{blue} $\frac{1}{5}$}\put(-38,20){\tiny\color{blue} $\frac{2}{5}$}\put(-19,25){\tiny\color{blue} $\frac{1}{5}$} 

\vspace{15pt}  
 \epsfig{file=./pic/ac4-1.eps,width=0.14\linewidth}\put(-19,25){\tiny\color{blue} $\frac{2}{5}$}
  \quad
  \raisebox{15pt}{\epsfig{file=./pic/arrow.eps,width=0.05\linewidth}}
  \quad
  \epsfig{file=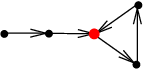,width=0.14\linewidth}\put(-19,25){\tiny\color{blue} $\frac{1}{5}$}\put(-19,0){\tiny\color{blue} $\frac{2}{5}$}  
  \quad
    \raisebox{15pt}{\epsfig{file=./pic/arrow.eps,width=0.05\linewidth}}
  \quad
  \epsfig{file=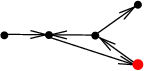,width=0.14\linewidth}\put(-19,25){\tiny\color{blue} $\frac{1}{5}$}\put(-12,11){\tiny\color{blue} $\frac{2}{5}$}\put(-28,-1){\tiny\color{blue} $\frac{2}{5}$} 
  \quad
    \raisebox{15pt}{\epsfig{file=./pic/arrow.eps,width=0.05\linewidth}}
  \quad
   \epsfig{file=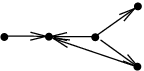,width=0.14\linewidth}\put(-19,25){\tiny\color{blue} $\frac{1}{5}$}\put(-12,11){\tiny\color{blue} $\frac{2}{5}$}\put(-28,-1){\tiny\color{blue} $\frac{2}{5}$}\put(-33,20){\tiny\color{blue} $\frac{2}{5}$}

\caption{Mutation equivalences between exceptional mutation-infinite quivers }
\label{mut}
\end{figure}

\end{document}